\documentclass[10pt]{amsart}
\usepackage{geometry}                
\usepackage{graphicx}
\usepackage{caption}
\usepackage{subcaption}
\usepackage{dsfont}
\usepackage{amssymb}
\usepackage{amsmath}
\usepackage{epstopdf}
\usepackage{fullpage}
\usepackage{enumerate}
\usepackage{color}
\usepackage{amsthm}
\usepackage{placeins}
\usepackage{array}
\usepackage{blkarray}
\usepackage{multirow}
\usepackage{float}
\usepackage{morefloats}
\usepackage{pgfplots}
\usepackage[ruled]{algorithm2e}
\usepackage{algorithmic}
\usepackage[abs]{overpic}
\usepackage{newclude}

\usepackage{amsopn}
\let\oldtocsection=\tocsection

\let\oldtocsubsection=\tocsubsection

\let\oldtocsubsubsection=\tocsubsubsection

\renewcommand{\tocsection}[2]{\hspace{0em}\oldtocsection{#1}{#2}\textbf}
\renewcommand{\tocsubsection}[2]{\hspace{1em}\oldtocsubsection{#1}{#2}}
\renewcommand{\tocsubsubsection}[2]{\hspace{2em}\oldtocsubsubsection{#1}{#2}}

\usepackage[colorlinks=true]{hyperref}
\usepackage[nameinlink,capitalise]{cleveref}
\crefname{equation}{}{}

\makeatletter
\def\subsection{\@startsection{subsection}{3}%
  \z@{.5\linespacing\@plus.7\linespacing}{.5\linespacing}%
  {\bf}}
\makeatother

\makeatletter
\def\subsubsection{\@startsection{subsubsection}{3}%
  \z@{.5\linespacing\@plus.7\linespacing}{.5\linespacing}%
  {\it}}
\makeatother

\newcommand{\dv}{\text{\rm div}}

\renewcommand{\o}{\text{\rm o}}

\renewcommand{\d}{\text{\rm d}}
\newcommand{\capa}{\text{\rm cap}}
\newcommand{\capar}{\text{\rm cap}_{\text{\rm Rob}}}

\newcommand{\e}{\varepsilon}
\newcommand{\vf}{\varphi}

\newcommand{\Eint}{E_{\text{\rm int}}}
\newcommand{\Eext}{E_{\text{\rm ext}}}

\newcommand{\calM}{{\mathcal M}}

\newcommand{\calO}{{\mathcal O}}
\newcommand{\calC}{{\mathcal C}}

\newcommand{\R}{{\mathbb R}}

\definecolor{darkgreen}{RGB}{1,150,32}

\newcommand{\D}{{\mathbb D}}

\newcommand{\ds}{\displaystyle}

\begin{document}
\newtheorem{theorem}{Theorem}[section]
\newtheorem{remark}{Remark}[section]
\newtheorem{definition}{Definition}[section]
\newtheorem{lemma}{Lemma}[section]
\newtheorem{corollary}{Corollary}[section]
\newtheorem{proposition}{Proposition}[section]
\newtheorem{propdef}{Definition-Proposition}[section]
\newtheorem{example}{Example}[section]
\numberwithin{equation}{section}

\title{Asymptotic expansion of the voltage potential under small Robin perturbations of its boundary conditions}
\author{
E. Bonnetier \textsuperscript{1}, 
C. Dapogny \textsuperscript{2},
R. Moskalenko \textsuperscript{1}
}
\maketitle
\begin{center}
\emph{\textsuperscript{1} Institut Fourier, Universit\'e Grenoble-Alpes, BP 74, 38402 Saint-Martin-d’H\`eres Cedex, France}. \\
\emph{\textsuperscript{2} Sorbonne Universit\'e, Universit\'e Paris Cit\'e, CNRS, Inria, Laboratoire Jacques-Louis Lions, LJLL, F-75005 Paris, France}.
\end{center}

 
\begin{abstract}
Inspired by questions related to inverse problems and shape optimization,
we derive an asymptotic expansion of the voltage potential, solution to a model elliptic second-order partial differential equation, under small perturbations of its boundary conditions. 
More precisely, the homogeneous Dirichlet or homogeneous Neumann boundary condition in a fixed, reference configuration of the problem is replaced by a Robin boundary condition with admittance $k_\e > 0$ on a ``small'' subset $\omega_\e$ of the boundary of the ambient domain, vanishing at the limit $\e \to 0$. In each of these two situations, a general asymptotic formula is established for the voltage potential, which rests on minimal assumptions about the shape of the vanishing subset $\omega_\e$ and the parameter $k_\e$. The scalings of these expansions, capturing the intensity of the perturbation, are measured by new quantities called  ``Robin-Dirichlet'' capacity or  a ``Robin-Neumann'' capacity, which depend on the geometry of $\omega_\e$ and on the value of the parameter $k_\e$. We analyze how these quantities compare to 
more classical measures of the ``smallness'' of $\omega_\e$, such as the capacity or the Neumann capacity of $\omega_\e$, according to the behavior of $k_\e$ as $\e\to 0$. 
\end{abstract} 

\medskip




\par
\vspace*{10mm}

\hrule
\vspace*{10mm}

\tableofcontents
\vspace*{5mm}
\hrule
\vspace*{5mm}

\section{Introduction}

\noindent The analysis  of the asymptotic effect of ``small'' perturbations of the features of a physical model, accounted for by a partial differential equation, is handful for a variety of purposes. 
On the one hand, it can be used to make the system robust with respect to uncertain fluctuations of its parameters \cite{bonnans2013perturbation}. On the other hand, this sensitivity is the basic ingredient of multiple inverse problems techniques for the gradient-based reconstruction of such features from observational data \cite{ammari2004reconstructionbook,isakov2006inverse}, or for the optimization of the shape of the system with respect to a performance criterion \cite{henrot2018shape,novotny2012topological}.

The present article aims to analyze one particular class of such perturbations. To set ideas, we consider throughout the following reference elliptic boundary value problem posed in a smooth bounded domain $\Omega$ in the Euclidean space $\R^d$, where $d=2$ or $3$:
\begin{equation}\label{eq.conducbg}
\left\{ \begin{array}{cccl}
\dv( \gamma \nabla u_0) &=& f & \textrm{in}\; \Omega,\\
u_0 &=& 0 & \textrm{on}\; \Gamma_D,\\
\gamma \frac{ \partial u_0}{\partial n} &=& 0 & \textrm{on}\; \Gamma_N.
\end{array}\right.
\end{equation}
Here, $\Gamma_D$ and $\Gamma_N$ are disjoint and complementary parts of $\partial \Omega$, respectively bearing homogeneous Dirichlet and Neumann boundary conditions and $f$ is a smooth source term.
The conductivity $\gamma$ of the medium is a smooth function satisfying:
\begin{equation} \label{cond_ell}
0 \;<\; \gamma^- \;\leq\; \gamma(x) \;\leq\; \gamma^+ < \infty, \quad \text{ for a.e. }\; x \in \Omega,
\end{equation}
for some fixed positive constants $\gamma^-$, $\gamma^+$. 
\par\medskip

\noindent \textbf{Volumic perturbations.} A large body of work in the literature is devoted to the analysis of the asymptotic behavior of the solution
to an elliptic equation of the form \cref{eq.conducbg} when its material coefficient $\gamma$ is perturbed in a small subset $\omega_\e$ of 
the domain $\Omega$. In the seminal paper \cite{cedio1998identification}, a perturbed version of \cref{eq.conducbg} is considered where $\gamma$ is replaced with another coefficient inside a diametrically small subset $\omega_\e$. Precisely, the conductivity $\gamma$ in \cref{eq.conducbg} is replaced with that $\gamma_\e$ defined by:
$$
\gamma_\e(x) \: = \: \gamma(x) + (\widetilde{\gamma}(x) - \gamma(x))\,\mathds{1}_{\omega_\e}(x), \quad x \in \Omega,
$$
where $\widetilde{\gamma}$ is another smooth function also satisfying \cref{cond_ell} and $\mathds{1}_A$ stands for the characteristic function of a set $A$.
The set  $\omega_\e := x_0 + \e \omega$ is a translated and rescaled version of a fixed inclusion pattern $\omega \subset \R^d$; $\omega_\e$ lies uniformly far from the boundary $\partial \Omega$ in the sense that;
$$ \forall \e >0, \quad \textrm{dist}(\omega_\e, \partial \Omega) \: \geq \: d_0 > 0,$$
for some fixed distance $d_0 > 0$.
In this context, it is shown in \cite{cedio1998identification} that at a point $x$ lying ``far'' from all the subsets
$\omega_\e$, the perturbed solution $u_\e$ satisfies
\begin{equation} \label{asymp_vol}
u_\e(x) \:\: =\:\: u_0(x) \;+\; |\omega_\e| \calM\nabla u_0(x_0) \cdot \nabla_y G(x,x_0)
\;+\; \o( |\omega_\e| ),
\end{equation}
where $G(x,y)$ is the Green's function of the background problem \cref{eq.conducbg}, see \cref{sec_2}.
The polarization matrix $\calM \in \R^{d\times d}$ in this formula is defined by: 
$$\calM_{ij} \;=\; \ds\int_\omega \frac{\partial \vf_j}{\partial y_i}(y) \, \d y, \quad i,j=1,\ldots,d, $$ 
involving the solutions $\varphi_j$ to $d$ auxiliary problems set in the infinite space $\R^d$:
$$-\dv\Big( \gamma_{1}(y) \nabla(\vf_j(y) + y_j) \Big) \;=\; 0 \quad\textrm{in}\; \R^d, \text{ where } \gamma_1(y) = \gamma(x_0) +  (\widetilde{\gamma}(x_0) - \gamma(x_0))\,\mathds{1}_{\omega}(y).$$

Asymptotic expansions of the form \cref{asymp_vol} were derived in more challenging physical contexts, described by e.g. the linear elasticity system \cite{ammari2002complete},
or the Helmholtz and Maxwell equations~\cite{ammari2001asymptotic}, 
and in the case of eigenvalue problems~\cite{ammari_moskow_eigen}. These analyses were also extended to the situations of
cracks or strip-like perturbations \cite{beretta2003asymptotic,beretta2001asymptotic,friedman1989identification}.
On a different note, assuming both conductivities $\gamma$ and $\widetilde\gamma$ to be constant, layer potential techniques allow to derive higher-order 
expansions of the perturbed field $u_\e$, see~\cite{ammari2007polarization} and the references therein.
\medskip

Intuitively, formula \cref{asymp_vol} expresses that, up to first order terms in $|\omega_\e|$,
the set~$\omega_\e$  asymptotically acts like a dipole placed at $x_0$, 
with moment $\calM \nabla u_0(x_0)$.
In the context of inverse problems, this formula contains the information that one may expect to recover 
from measurements of $u_\e - u_0$: 
the singularity of $\nabla_y G(x,x_0)$ allows to determine the position $x_0$
of the center of $\omega_\e$, while some partial information about the geometry of $\omega$ and on the contrast between material coefficients
is encoded in $\calM$.
Asymptotic expansions of the form \cref{asymp_vol} have indeed been used to design efficient detection algorithms akin to 
linear sampling methods~\cite{AmmariIakovlevaLesselierPerrusson,bruhl2003direct},
and to propose approximate cloaking mechanisms \cite{Kohn_Shen_Vogelius_Weinstein}.
\medskip

The general structure of the first-order correction to the background potential induced by a perturbation of the conductivity coefficient within a ``small'' set $\omega_\e$ of arbitrary shape was elucidated in \cite{capdeboscq2003general}.
This work proceeds under the following minimal assumptions about $\omega_\e$:
\begin{itemize}
\item[-] For each $\e >0$, $\omega_\e$ may be the reunion of  $n_\e < \infty$ 
connected components,
\medskip

\item[-] The $\omega_\e$ are all contained in a fixed compact set $K_0 \Subset \Omega$,
and thus stay uniformly away from $\partial \Omega$,
\medskip

\item[-] The volume $|\omega_\e|$ tends to $0$ as $\e \to 0$.
\end{itemize}
In this setting, there exists a subsequence of the indices $\e$, still denoted by $\e$ for simplicity, 
a non-trivial Radon measure~$\mu$ with support in~$K_0$, 
and a polarization tensor field~${\mathcal M} \in L^2(\Omega;\d \mu)$ 
such that the following asymptotic expansion holds true at each point $x \in \Omega \setminus K_0$:
\begin{equation} \label{asympYM}
u_\e(x) \:\: = \:\: u_0(x) + |\omega_\e| \langle \mu(y), {\mathcal M}(y) \nabla u_0(y) \cdot \nabla G(x,y) \rangle
\;+\; \o(|\omega_\e|).
\end{equation} 
This compensated compactness result echoes to typical results from two-phase homogenization theory \cite{allaire2002shape,bensoussan2011asymptotic,murat2018h}; it somehow shows that ``small'' perturbations are ``dilute'' limits of the homogenization phenomenon, i.e. when the fraction of one of the phases at play tends to $0$.

Counterpart structure theorems were later proved in the more challenging physical contexts of elasticity and electromagnetism, see \cite{beretta2012small} and \cite{griesmaier2011general}, respectively. From the applications viewpoint, these developments have led to the notion of topological derivative, appraising how a physical quantity, involving the voltage potential, depends on the introduction of a small heterogeneity (or a hole) within the ambient medium, see \cite{dapogny2021topological,garreau2001topological,novotny2012topological,sokolowski1999topological}.
 \par\medskip

\noindent \textbf{Perturbations of boundary conditions.} 
Our recent work \cite{bonnetier_dapogny_vogelius} deals with another type of perturbations of the model background equation \cref{eq.conducbg}, which consists in altering its boundary conditions on a ``small'' region $\omega_\e$ of $\partial \Omega$.
This work considers two different situations:
\begin{itemize}
\item[(i)]
The sets $\omega_\e$ are contained in the region $\Gamma_N$ and the Neumann boundary condition imposed in there is replaced by a homogeneous Dirichlet condition, leading to a perturbed potential $u_\e^{N,\infty}$.
\item[(ii)]
The sets $\omega_\e$ are contained in $\Gamma_D$ and the Dirichlet condition imposed in there is replaced by a homogeneous Neumann condition, leading to a perturbed potential $u_\e^{D,0}$.
\end{itemize}
In both situations, the perturbed voltage potential 
has a similar asymptotic structure to that in \cref{asympYM}. Precisely,
up to extraction of a subsequence (still indexed by $\e$), there exists a Radon measure $\mu$ 
supported on~$\partial \Omega$, such that for any smooth cut-off function $\eta$, with support strictly inside $\Gamma_N$, it holds:
\begin{equation}\label{asymp_pertN}
u_\e^{N,\infty}(x) \: = \: u_0(x) + \capa(\omega_\e) \left\langle \mu(y), \eta(y) u_0(y) \gamma(y) G(x,y) \right \rangle
\;+\; \o\big(\capa(\omega_e)\big),
\end{equation}
and for any smooth cut-off function $\eta$, whose support is now strictly contained in $\Gamma_D$:
\begin{equation}\label{asymp_pertD}
u_\e^{D,0}(x) \: = \: u_0(x) + e(\omega_\e) \left\langle\mu(y), \eta(y)  \gamma(y)\frac{\partial u_0}{\partial n}(y)  \frac{\partial G}{\partial n_y} (x,y) \right\rangle
\;+\; \o\big(e(\omega_\e)\big).
\end{equation}
The scalings featured in the asymptotic expressions \cref{asymp_pertN,asymp_pertD} are different. 
The replacement of a homogeneous Neumann by a Dirichlet boundary condition impacts the voltage potential
by a term of the order of the capacity of $\omega_\e$, which is defined by \cite{henrot2018shape}:
\begin{equation}\label{eq.defcapa}
\capa(\omega_\e) =
\inf \left\{ ||v||^2_{H^1(\R^d)}, \:\: v \in H^1(\R^d), \:\: v = 1\;
\textrm{in a neighborhood of}\; \omega_\e \right\}.
\end{equation}
In the second situation \cref{asymp_pertD}, the difference $u_\e^{D,0} - u_0$ scales like the ``Neumann capacity''
$e(\omega_\e)$ defined in our previous work \cite{bonnetier_dapogny_vogelius} as:
\begin{eqnarray*}
e(\omega_\e) &=& \inf \ds\int_{\R^d \setminus \overline{\omega_\e}} \Big( \lvert \nabla z\lvert ^2 + z^2 \Big) \:\d x,
\end{eqnarray*}
where the infimum is taken over all functions $z \in H^1(\R^d \setminus \overline{\omega_\e})$
that solve
\[ \left\{ \begin{array}{ccll}
-\Delta z + z &=& 0 & \textrm{in}\; \R^d \setminus \overline{\omega_\e},
\\
\frac{\partial z}{\partial n} &=& \pm 1 & \textrm{on each of the connected components of}\; \omega_\e.
\end{array}\right.
\]
For particular geometries of the sets $\omega_\e$,
the measure $\mu$ can be fully explicited. 
These results apply to shape optimization problems where one seeks to
optimize subsets of the boundary $\partial \Omega$ where particular boundary conditions
are imposed: in~\cite{bonnetier_brito_dapogny_estevez}, the asymptotic 
formulas \cref{asymp_pertN} and \cref{asymp_pertD} are used to calculate ``topological 
derivatives'' for a quantity of interest to be minimized that depends on the voltage potential, and a gradient descent strategy is implemented from this datum.

 \par\medskip
\noindent \textbf{Contents of this article.} 
The present work elaborates on the analysis of \cite{bonnetier_dapogny_vogelius}
and considers the replacement of a homogeneous Neumann or Dirichlet boundary condition in the background problem \cref{eq.conducbg} by a Robin boundary condition with admittance parameter $k_\e$. 
More precisely, we consider two different types of perturbations of the problem \cref{eq.conducbg}:
\begin{enumerate}[(i)]
\item Introducing a family of subsets $\omega_\e \Subset \Gamma_N$, we denote by $u_\e^{N,k_\e}$ the solution to the perturbation of \cref{eq.conducbg} where the homogeneous Neumann boundary condition is replaced by a Robin condition with admittance parameter $k_\e$ on the subset $\omega_\e$, see \cref{eq_pertN} below.
\item We consider the perturbation of \cref{eq.conducbg} where the homogeneous Dirichlet boundary condition is replaced by a Robin condition with parameter $k_\e$ on a subset $\omega_\e \subset \Gamma_D$; the perturbed potential $u_\e^{D,k_\e}$ is then the solution to \cref{eq_pertD} below. 
\end{enumerate}
In both situations, the Robin parameter $k_\e > 0$ may depend on $\e$.
Formally, the Robin boundary condition ``interpolates'' between Dirichlet and Neumann conditions:
when $\e$ is fixed, one expects that when $k_\e$ is large, $u_\e^{N,k_\e}$ approaches a function 
that satisfies a homogeneous Dirichlet boundary condition on $\omega_\e$, 
while $u_\e^{D,k_\e}$ should approach a function that satisfies a homogeneous Neumann condition on 
$\omega_\e$ as $k_\e \to 0$. We refer to e.g. \cite{engelstein2026robin} about the physical interpretation of the Robin boundary condition. 

The purpose of this article is to derive asymptotic expansions of the differences $u_\e^N - u_0$ and $u_\e^D - u_0$ between the reference and perturbed voltage potentials by the replacement of Neumann or Dirichlet by Robin boundary conditions with coefficient $k_\e$ in a vanishingly small subset $\omega_\e \subset \partial \Omega$. 
These expansions express the magnitude of these differences in terms of new quantities $\capar^N(\omega_\e,k_\e)$, $\capar^D(\omega_\e,k_\e)$ depending on the subset $\omega_\e$ and on the admittances $k_\e$, that deserve the names ``Robin capacities''. 
We further investigate how these asymptotic measures of smallness transition between the more usual quantities
$e(\omega_\e)$ and $\capa(\omega_e)$, depending on the behavior of the parameter sequence $k_\e$.

This work is related to that in \cite{li2014matched}, where a matched asymptotic expansion methods is used to conduct the asymptotic analysis of a Robin perturbation within a vanishing subset $\omega_\e$ of a Neumann boundary with a fixed admittance parameter in 2d. This study is extended to 3d in \cite{li2017asymptotic}, where layer potential techniques are used. Let us also mention work on homogenization of periodic networks
of inhomogeneities with Robin boundary conditions with high contrast, see~\cite{GomezPerezShaposhnikova} 
and the references therein.
\medskip

\noindent \textbf{Organization of the article.} 
In the next \cref{sec_2}, we describe the setup and recall some necessary background material, before introducing the two notions of Robin-Neumann and Robin-Dirichlet capacities $\capar^N(\omega_\e,k_\e)$ and $\capar^D(\omega_\e,k_\e)$.
\cref{sec.resstate} is then devoted to the statement of our main results, firstly when the Neumann region $\Gamma_N$ is perturbed, 
secondly when the Dirichlet one $\Gamma_D$ is perturbed.
In each case we describe the asymptotic structure of the voltage potential and its proper scaling, see \cref{thm1,thm2}. 
We next compare the Robin capacities $\capar^N(\omega_\e,k_\e)$
and $\capar^D(\omega_\e,k_\e)$ to those obtained in our previous work~\cite{bonnetier_dapogny_vogelius}
when a homogeneous Dirichlet or homogeneous Neumann condition was imposed on~$\omega_\e$
instead of a Robin condition. 
In particular, we show in \cref{thm5}
that the transition from $\capar^N(\omega_\e,k_\e)$ to $\capa(\omega_\e)$ is continous
as $k_\e$ tends to infinity sufficiently fast.
A similar result, whereby $\capar^D(\omega_\e,k_\e)$ tends to $e(\omega_\e)$
as $k_\e \to 0$ is the subject of \cref{thm4}.
\cref{sec_3,sec_4} are devoted to the proofs of these results
in the case where the perturbations bear on $\Gamma_N$, while
\cref{sec_5} concerns the case of perturbations impinging on $\Gamma_D$.
In \cref{sec_6}, we introduce equivalent measures of the 
Robin capacities $\capar^{N}(\omega_\e,k_\e)$ and $\capar^{D}(\omega_\e,k_\e)$, that are intrinsic to the perturbation sets $\omega_\e$, 
and in particular that do not depend on the rest of $\partial \Omega$.
The article ends with an appendix dedicated to a technical lemma involved in the proof
of \cref{thm3} and that is of interest in its own right.


\section{Setting of the problem and preliminaries} \label{sec_2}


\noindent This section introduces the physical setting of the article and sets the main notation used throughout. 
After the formal statement of the problems under scrutiny in \cref{sec.statepb}, we recall basic definitions about fractional Sobolev spaces in \cref{sec.Hs}. Then, in \cref{sec.capaDN}, we recall a few results about 
the notions of capacity related to perturbations by Dirichlet or Neumann boundary conditions, before finally introducing the quantities adapted to account for the smallness of perturbations by Robin boundary conditions in \cref{sec.capaR}.

\subsection{Statement of the problems}\label{sec.statepb}

\noindent Let $\Omega$ be a smooth bounded domain in $\R^d$ ($d=2$ or $3$), 
whose boundary is made of two disjoint regions with non-empty interiors:
$$ \partial \Omega = \overline{\Gamma_D} \cup \overline{\Gamma_N}.$$
Here, $\Gamma_D$ and $\Gamma_N$ are relatively open subsets of $\partial \Omega$ such that:
\begin{itemize}
\item Homogeneous Dirichlet boundary conditions are applied on $\Gamma_D$,
\item Homogeneous Neumann boundary conditions are applied on $\Gamma_N$.
\end{itemize}
Assuming the presence of a smooth source $f\in \calC^\infty(\overline\Omega)$, 
the voltage potential $u_0$ within $\Omega$ in this reference, ``background'' situation is the 
unique solution in the space 
$$H^1_{\Gamma_D}(\Omega) := \left\{ u \in H^1(\Omega), \:\: u=0 \text{ on } \Gamma_D \right\}, $$
to the conductivity equation \cref{eq.conducbg}, that is reproduced below for convenience: 
\begin{equation}\label{eq.conducbg2}
\left\{ \begin{array}{cccl}
\dv( \gamma \nabla u_0) &=& f & \textrm{in}\; \Omega,
\\
u_0 &=& 0 & \textrm{on}\; \Gamma_D,
\\
\gamma \frac{ \partial u_0}{\partial n} &=& 0 & \textrm{on}\; \Gamma_N,
\end{array}\right.
\end{equation}
where the conductivity 
$\gamma \in \calC^\infty(\overline{\Omega})$ is elliptic, in the sense that it satisfies \cref{cond_ell}.
Let us recall that, as a classical consequence of the elliptic regularity theory, the function $u_0$ is smooth except at the points of $\overline{\Gamma_D} \cap \overline{\Gamma_N}$, where the boundary conditions change types in \cref{eq.conducbg2}:
$$\text{For all } x \in \overline{\Omega} \setminus \Big(\overline{\Gamma_D} \cap \overline{\Gamma_N} \Big), \text{ there exists an open neighborhood } \calO \text{ of } x \text{ s.t. } u_0 \in \calC^\infty(\overline\Omega \cap \calO),  $$
see e.g. \cite{brezis2010functional,gilbarg2015elliptic}. 

The background potential $u_0$ can be represented in terms of the Green's function $G(x,y)$ associated to the reference configuration \cref{eq.conducbg2}. The latter is defined as follows: for each fixed point $x \in \Omega$, the function $y \mapsto G(x,y)$ is the solution to:
\begin{equation} \label{eq_Green}
\left\{ \begin{array}{ccll}
-\dv_y( \gamma(y) \nabla_y G(x,y)) &=& \delta_x(y) &\textrm{in}\; \Omega,
\\[5pt]
G(x,y) &=& 0 & \textrm{on}\; \Gamma_D,
\\[5pt]
\gamma(y) \frac{\partial G}{\partial n_y}(x,y) &=& 0
&\textrm{on}\; \Gamma_N;
\end{array} \right.
\end{equation}
we refer e.g. to \cite{ammari2007polarization,folland1995introduction,friedman1989determining} for the construction of such functions.
It then holds:
$$ u_0(x) = \int_\Omega G(x,y) f(y) \:\d y, \quad x \in \Omega.$$

As mentioned in the introduction, we analyze two different types of perturbations of the background problem \cref{eq.conducbg}: 
\begin{enumerate}[(i)]
\item Let $\omega_\e$ be a family of subsets of the boundary $\partial\Omega$, indexed by $\e >0$, which are ``well-inside'' the Neumann region $\Gamma_N$, that is: 
\begin{equation}\label{eq.assumsepKN}
\text{There exists a fixed compact subset } K_N \subset \Gamma_N \text{ s.t. } \forall \e >0, \quad \omega_\e \subset K_N.
\end{equation}
The homogeneous Neumann boundary condition in \cref{eq.conducbg} is replaced by a Robin boundary condition on $\omega_\e$, and the perturbed potential $u_\e^{N,k_\e}$ is the $H^1(\Omega)$ solution to: 
\begin{equation}\label{eq_pertN}
\left\{ \begin{array}{cccl}
\dv( \gamma \nabla u_\e^{N,k_\e} ) &=& f & \textrm{in}\; \Omega,
\\[5pt]
u_\e^{N,k_\e} &=& 0 & \textrm{on}\; \Gamma_D,
\\[5pt]
\gamma(x) \frac{\partial u_\e^{N,k_\e}}{\partial n} &=& 0 & \textrm{on}\; \Gamma^N \setminus \overline{\omega_\e},
\\[5pt]
\gamma  \frac{\partial u_\e^{N,k_\e}}{\partial n} + k_\e u_\e^{N,k_\e} &=& 0 & \textrm{on}\; \omega_\e.
\end{array}\right.
\end{equation}
\item Let $\omega_\e$ be a family of subsets of $\partial \Omega$ which are ``well-inside'' the Dirichlet region $\Gamma_D$: 
\begin{equation}\label{eq.assumsepKD}
\text{There exists a fixed compact subset } K_D \subset \Gamma_D \text{ s.t. } \forall \e >0, \quad \omega_\e \subset K_D.
\end{equation}
The homogeneous Dirichlet boundary condition in \cref{eq.conducbg} is replaced by a Robin boundary condition on $\omega_\e$, and the perturbed potential $u_\e^{D,k_\e}$ is the $H^1(\Omega)$ solution to: 
\begin{equation}\label{eq_pertD}
\left\{ \begin{array}{cccl}
\dv( \gamma \nabla u_\e^{D,k_\e}) &=& f, & \textrm{in}\; \Omega,
\\[5pt]
u_\e^{D,k_\e} &=& 0 & \textrm{on}\; \Gamma_D \setminus \overline{\omega_\e},
\\[5pt]
\gamma  \frac{\partial u_\e^{D,k_\e}}{\partial n} &=& 0 & \textrm{on}\; \Gamma_N,
\\[5pt]
\gamma \frac{\partial u_\e^{D,k_\e}}{\partial n} + k_\e u_\e^{D,k_\e} &=& 0 & \textrm{on}\; \omega_\e.
\end{array}\right.
\end{equation}
\end{enumerate}

Both problems \cref{eq_pertN,eq_pertD} are well-posed, under the assumption that for each~$\e~>~0$, 
the admittance $k_\e$ is positive and finite.
However, we make no assumption about the behavior of $k_\e$ as $\e \to 0$. In particular, it may very well tend to $0$ or to $\infty$ in this limit.


\subsection{Fractional Sobolev spaces on boundary regions}\label{sec.Hs}

\noindent This section gathers classical definitions about fractional Sobolev spaces defined on (a subset of) the boundary of a smooth domain. This material is mainly excerpted from \cite{DINEZZA2012521,grisvard2011elliptic,mclean2000strongly}. 

Let us first consider Sobolev spaces defined on the total boundary $\partial \Omega$ of a smooth bounded domain $\Omega$ in $\R^d$.
For $0 < s< 1$, the space $H^s(\partial \Omega)$ is defined by: 
\begin{multline}\label{eq.HspOm}
 H^s(\partial \Omega) = \left\{ u \in L^2(\partial \Omega), \:\: \lvert\lvert u \lvert\lvert_{L^2(\partial\Omega)}^2 + \lvert u \lvert^2_{H^s(\partial \Omega)} < \infty \right\}, \text{ where } \\
  \lvert u \lvert^2_{H^s(\partial \Omega)} := \int_{\partial \Omega} \int_{\partial \Omega} \frac{\lvert u(x) - u(y) \lvert^2}{\lvert x - y \lvert^{d-1+2s}}\:\d s(x) \d s(y).
\end{multline}
This space is a Hilbert space, with norm:
$$\vert\vert u \lvert\lvert_{H^s(\partial \Omega)} = \Big(\lvert\lvert u \lvert\lvert^2_{L^2(\partial \Omega)} + \lvert u \lvert^2_{H^s(\partial \Omega)}\Big)^{1/2}.$$
Still for $0< s < 1$, the negative index Sobolev space $H^{-s}(\partial\Omega)$ is the topological dual of $H^s(\partial \Omega)$.

Let now $\omega$ be an open Lipschitz subset of $\partial \Omega$. For $0 < s < 1$, the space $H^s(\omega)$ is defined in a similar manner to \cref{eq.HspOm}: 
$$ H^s(\omega) = \left\{ u \in L^2(\omega), \:\: \lvert\lvert u \lvert\lvert_{L^2(\omega)}^2 + \lvert u \lvert^2_{H^s(\omega)} < \infty \right\}, \text{ where }  \lvert u \lvert^2_{H^s(\omega)} := \int_{\omega} \int_{\omega} \frac{\lvert u(x) - u(y) \lvert^2}{\lvert x - y \lvert^{d-1+2s}}\:\d s(x) \d s(y).$$
Again, $H^s(\omega)$ is a Hilbert space when equipped with the norm  
\begin{equation}\label{eq.normHsom}
\vert\vert u \lvert\lvert_{H^s(\omega)} = \Big(\lvert\lvert u \lvert\lvert^2_{L^2(\omega)} + \lvert u \lvert^2_{H^s(\omega)}\Big)^{1/2}.
\end{equation}
It is well-known that any function in $H^s(\omega)$ can be extended into a function in $H^s(\partial \Omega)$, and that  $\lvert\lvert \cdot \lvert\lvert_{H^s(\omega)}$ is equivalent to the quotient norm:
\begin{equation}\label{eq.Hsnormquotient}
 u  \:\: \longmapsto \:\:  \inf\Big\{\lvert\lvert v \lvert\lvert_{H^s(\partial \Omega)} , \:\: v \in H^s(\partial \Omega) \text{ s.t. } v =u \text{ on } \omega \Big\},
\end{equation}
the equivalence constants between both norms being dependent on $\omega$ and $\partial \Omega$.

The situation of a strict subset $\omega \subset \partial \Omega$ features another type of Sobolev space of interest for our purpose, namely that $\widetilde{H}^{s}(\omega)$ gathering the elements in $H^s(\omega)$ whose extension to $\partial \Omega$ by $0$ belongs to $H^{s}(\partial \Omega)$. This space is equipped with the norm:
$$ \lvert\lvert u \lvert\lvert_{\widetilde{H}^s(\omega)} =  \lvert\lvert \widetilde{u} \lvert\lvert_{H^s(\partial \Omega)}, \text{ where } \widetilde u = \left\{
\begin{array}{cl}
u & \text{on } \omega, \\
0 & \text{on } \partial \Omega \setminus \overline\omega.
\end{array}
\right.$$

Let us eventually discuss Sobolev spaces with negative index on the subset $\omega$ of $\partial \Omega$. For $0 < s < 1$, $H^{-s}(\omega)$ is still defined as the space of the restrictions to $\omega$ of distributions in $H^{-s}(\partial \Omega)$, and it is again equipped with the quotient norm, see \cref{eq.Hsnormquotient}. Equivalently, $H^{-s}(\omega)$ is an isometric realization of the dual of $\widetilde{H}^s(\omega)$, through the following pairing:
$$\text{For all } u \in H^{-s}(\omega) , \: v \in \widetilde{H}^s(\omega), \quad \langle u,v \rangle_{H^{-s}(\omega), \widetilde{H}^s(\omega)} = \langle U , \widetilde v \rangle_{H^{-s}(\omega), \widetilde{H}^s(\omega)}, $$
where $U$ is an arbitrary element in $H^{-s}(\partial \Omega)$ such that $U \lvert_{\omega} = u$ and $\widetilde v \in H^s(\partial \Omega)$ is the extension of $v$ by $0$ to $\partial \Omega$.
Likewise, $\widetilde H^{-s}(\omega)$ is defined as the space of distributions in $H^{-s}(\partial \Omega)$ with compact support inside $\overline\omega$. It is the dual space of $H^s(\omega)$ through the pairing:
$$\text{For all } u \in \widetilde{H}^{-s}(\omega) , \: v \in H^s(\omega), \quad \langle u,v \rangle_{\widetilde{H}^{-s}(\omega), H^s(\omega)} = \langle \widetilde u , V \rangle_{\widetilde{H}^{-s}(\omega), H^s(\omega)}, $$
where $\widetilde u \in H^{-s}(\partial \Omega)$ is the extension of $u$ by $0$ to $\partial \Omega$ and $V$ is an arbitrary element in $H^s(\partial \Omega)$ such that $V \lvert_{\omega} = v$.
\par\medskip

\subsection{Capacity and equilibrium distributions attached to Dirichlet and Neumann boundary conditions}\label{sec.capaDN}

\noindent In this section, we summarize some classical material about the notion of capacity, as well as some findings from our previous work \cite{bonnetier_dapogny_vogelius} that are handful for our purpose; we refer the reader to the books \cite{adams2012function,doob1984classical,landkof1972foundations} about these concepts.\par\medskip

At first, let $\omega$ be a Lipschitz subset of the region $\Gamma_N$; the Dirichlet equilibrium potential $\chi^\infty = \chi^\infty(\omega)$ 
is the unique $H^1(\Omega)$ solution to the following boundary-value problem:
\begin{equation}\label{def_chi_inf}
\left\{
\begin{array}{cccl}
-\dv(\gamma\nabla \chi^\infty) &=& 0 & \text{in } \Omega, 
\\[5pt]
\chi^\infty &=& 0 & \text{on } \Gamma_D,
\\[5pt]
\gamma \frac{\partial \chi^\infty}{\partial n} &=& 0 & \text{on } \Gamma_N \setminus \overline{\omega},
\\[5pt]
 \chi^\infty &=& 1 & \text{on } \omega.
\end{array}
\right.
\end{equation}
Intuitively, $\chi^\infty$ is the function in the admissible space $H^1_{\Gamma_D}(\Omega)$ for \cref{eq.conducbg2} that nearly realizes the minimum value in the definition \cref{eq.defcapa} of capacity.\par\medskip 

In the same spirit, let now $\omega$ be a Lipschitz open subset of the region $\Gamma_D$. We define the 
Neumann equilibrium potential $\xi^0 = \xi^0(\omega)$ as the $H^1(\Omega)$ solution to:
\begin{equation}\label{eq.capfuncD}
\left\{
\begin{array}{cccl}
-\dv(\gamma\nabla \xi^0) &=& 0 & \text{in}\; \Omega, 
\\[5pt]
\xi^0 &=& 0 & \text{on } \Gamma_D \setminus \overline\omega,
\\[5pt]
\gamma \frac{\partial \xi^0}{\partial n} &=& 0 & \text{on}\; \Gamma_N,
\\[5pt]
\gamma \frac{\partial \xi^0}{\partial n}  &=& 1 & \text{on}\; \omega.
\end{array}
\right.
\end{equation}

When $\omega_\e$ is a collection of subsets of $\Gamma_N$ or $\Gamma_D$ indexed by a real number $\e$, we shall denote $\chi^\infty(\omega_\e) = \chi^\infty_\e$ and $\xi^0(\omega_\e) = \xi^0_\e$.

These capacity functions are related to the scalings that appear in the asymptotic expansions \cref{asymp_pertN} 
and \cref{asymp_pertD}. We indeed recall the following estimates from our previous work~\cite{bonnetier_dapogny_vogelius}.

\begin{proposition} \label{prop_1}
Let $\Omega \subset \R^d$ be a smooth bounded domain, whose boundary is made of two disjoint open regions $\Gamma_N$ and $\Gamma_D$ as in \cref{sec.statepb}, and let $K_N$, $K_D$ be two compact subsets of $\Gamma_N$ and $\Gamma_D$, respectively. Then,
\noindent \begin{enumerate}[(i)]
\item Let $\omega$ be a Lipschitz subset of $\Gamma_N$ lying far from $\Gamma_D$ in the sense that \cref{eq.assumsepKN} holds. There exists a constant $C > 0$, depending on $\Omega$, $\Gamma_N$, $\Gamma_D$ and the compact $K_N$ but not on $\omega$, such that:
$$
\frac{1}{C} \capa(\omega)^{1/2} \:\leq\: ||\chi^\infty||_{H^1(\Omega)} \;\leq\; C\capa(\omega)^{1/2}
\quad \textrm{and} \quad
||\chi^\infty||_{L^2(\Omega)} \;\leq\; C\, \capa(\omega)^{3/4}.
$$
\item Let $\omega$ be a Lipschitz subset of $\Gamma_D$ satisfying \cref{eq.assumsepKD}. There exists a constant $C > 0$, depending on $\Omega$, $\Gamma_N$, $\Gamma_D$ and the compact $K_D$ but not on $\omega$, such that:
$$\frac{1}{C} e(\omega)^{1/2} \:\leq\: ||\xi^0||_{H^1(\Omega)} \;\leq\; C e(\omega)^{1/2}
\quad \textrm{and} \quad
||\xi^0||_{L^2(\Omega)} \;\leq\; C\, e(\omega)^{3/4}.
$$
\end{enumerate}
\end{proposition}\par\medskip

\subsection{Capacity and equilibrium distributions attached to Robin boundary conditions}\label{sec.capaR}

\noindent This section introduces the equilibrium distributions and the associated notions of capacity 
involved in our analysis of Robin boundary conditions.

Let us first consider a subset $\omega \Subset \Gamma_N$ and let $k >0$. The Robin-Neumann potential  $\chi = \chi(\omega,k)$ is
defined as the unique $H^1(\Omega)$ solution to the following boundary-value problem:
\begin{equation}\label{def_chi}
\left\{
\begin{array}{cccl}
-\dv(\gamma\nabla \chi) &=& 0 & \text{in } \Omega, 
\\[5pt]
\chi &=& 0 & \text{on } \Gamma_D,
\\[5pt]
\gamma \frac{\partial \chi}{\partial n} &=& 0 & \text{on } \Gamma_N \setminus \overline{\omega},
\\[5pt]
\gamma \frac{\partial \chi}{\partial n}  + k \chi &=& k & \text{on } \omega.
\end{array}
\right.
\end{equation}
Equivalently, this function satisfies the following variational problem:
\begin{equation}\label{def_chivarf}
\text{Find}\; \chi \in H^1_{\Gamma_D}(\Omega) \text{ s.t. } \forall v \in H^1_{\Gamma_D}(\Omega) , \quad \int_\Omega \gamma \nabla \chi \cdot \nabla v \:\d x + k\int_{\omega} \chi v \:\d s 
= k \int_{\omega} v \:\d s. 
\end{equation}
Thence, the Robin-Neumann capacity $\capar^N(\omega,k)$ of the set $\omega$ is defined by
\begin{equation}\label{eq.defcapar}
\capar^N(\omega,k) = \int_{\omega} \gamma \frac{\partial \chi}{\partial n}\:\d s 
= k \int_{\omega} (1-\chi) \:\d s.
\end{equation}
\medskip


Let us now consider a subset $\omega \Subset \Gamma_D$ and let $k>0$. The Robin-Dirichlet 
potential $\xi = \xi(\omega,k)$ is the $H^1$ solution to the following boundary-value problem:
\begin{equation} \label{def_xi}
\left\{ \begin{array}{ccll}
- \dv(\gamma \nabla \xi)  &=& 0 & \textrm{in}\; \Omega,
\\[5pt]
\xi &=& 0 & \textrm{on}\; \Gamma_D \setminus \omega,
\\[5pt]
\gamma\frac{\partial \xi}{\partial n} &=& 0 & \textrm{on}\; \Gamma_N,
\\[5pt]
\gamma\frac{\partial \xi}{\partial n} + k \xi &=& 1 & \textrm{on}\; \omega,
\end{array} \right.
\end{equation}
or equivalently, in variational form
\begin{eqnarray*}
\forall v \in H^1_{\Gamma_D \setminus \overline\omega}(\Omega), \quad
\ds\int_\Omega \gamma \nabla \xi \cdot \nabla v \:\d x
\;+\; k \ds\int_{\omega} \xi v \:\d s
&=& \ds\int_{\omega} v \:\d s.
\end{eqnarray*}
\medskip
The Robin-Dirichlet capacity $\capar^D(\omega,k)$ of $\omega$ and $k > 0$ is then defined by
\begin{equation} \label{def_cRD}
\capar^D(\omega,k) \:\: = \:\: \ds\int_\Omega \gamma |\nabla \xi|^2 \:\d x
\;+\; k\, \ds\int_\omega \xi^2 \d s
\;=\;
\ds\int_\omega \xi \d s.
\end{equation}

Again, when $\omega$ is a collection $\omega_\e$ of Lipschitz subsets of $\Gamma_N$ or $\Gamma_D$ and $k$ accounts for a sequence $k_\e$ of real numbers, we simply write
$\chi(\omega_\e,k_\e) = \chi_\e$ and $\xi_\e(\omega_\e,k_\e) = \xi_\e$.

\begin{remark}
The above definitions \cref{eq.defcapar,def_cRD} of the Robin capacities $\capar^N(\omega,k)$ and $\capar^D(\omega,k)$ may look a little awkward at first glance, since they depend on the boundary-value problems \cref{def_chi,def_xi}, while, ideally, one would expect a measure of smallness that depends only on the Robin condition, i.e. on $\omega$ and $k$. As a matter of fact, \cref{sec_6} below remedies this conceptual drawback, as it proposes equivalent quantities (up to constants depending on $\Omega$) that are only defined in terms of $\omega$ and $k$. For technical convenience however, we stick to the definitions \cref{eq.defcapar,def_cRD} of the Robin capacities in our analysis.
\end{remark} 

\section{Statement of the main results}\label{sec.resstate}

\noindent This section gathers and comments the main findings of this article. At first, \cref{subsec.stateRobNeu} deals with the situation where the homogeneous Neumann boundary condition featured in the background problem \cref{eq.conducbg2} is replaced by a Robin boundary condition, and \cref{subsec.stateRobDir} is devoted to the alternative situation where  the homogeneous Dirichlet boundary condition is perturbed. 
The claims of this section are proved in the subsequent \cref{sec_3,sec_4,sec_5}. 

\subsection{Robin perturbation of the Neumann boundary condition on $\Gamma_N$} \label{subsec.stateRobNeu}

\noindent In this section, we consider the situation where the ``small'' perturbation set $\omega_\e$ satisfies $\omega_\e \Subset \Gamma_N$.
The perturbed version of the background problem \cref{eq.conducbg2} under scrutiny is \cref{eq_pertN}.\par\medskip 

Our first result shows that the Robin-Neumann capacity $\capar^N(\omega_\e,k_\e)$ is the adequate quantity to appraise the ``smallness'' of a Robin perturbation with parameter $k_\e$ on a subset $\omega_\e$ of $\Gamma_N$.

\begin{theorem}\label{thm1}
Let $\Omega$ be a smooth bounded domain of $\R^d$ whose boundary is divided into two disjoint open subsets $\Gamma_N$ and $\Gamma_D$, and let $K_N \subset \Gamma_N$ be a fixed compact set, see \cref{sec.statepb}.
Let $\omega_\e$ be a sequence of subsets of $\Gamma_N$ which is well-separated from $\Gamma_D$ in the sense that \cref{eq.assumsepKN} holds true, and let $k_\e$ be any sequence of positive real numbers. 
Assume that the Robin-Neumann capacity $\capar^N(\omega_\e,k_\e) $ in \cref{eq.defcapar} tends to $0$ as $\e \to 0$.
Then, there exists a subsequence, still denoted by $\e$, 
as well as a positive Radon measure $\mu$ with unit mass, supported on $\Gamma_N$, 
such that the following asymptotic expansion of the perturbed potential $u_\e^{N,k_\e}$ in \cref{eq_pertN} holds true:
\begin{equation}\label{asymp_uN}
\text{For } x \in \Omega, \quad u_\e^{N,k_\e}(x) =
u_0(x) - \capar^N(\omega_\e,k_\e) \, \langle \mu, u_0(\cdot) G(x,\cdot) \rangle
+ \o\left( \capar^N(\omega_\e,k_\e) \right).
\end{equation}
\end{theorem}

Our next result aims to better appraise the relation between the Robin capacity 
$\capar^N(\omega_\e,k_\e)$  and the classical capacity $\capa(\omega_\e)$ as $k_\e$ varies. Such a relation is to be expected, as $\capa(\omega_\e)$ is the scaling obtained 
in~\cite{bonnetier_dapogny_vogelius} when a mere Dirichlet condition is imposed 
on $\omega_\e$. In view of the asymptotic formula \cref{asymp_vol} for ``volumic'' perturbations of the conductivity $\gamma$, we are also interested in how $\capar^N(\omega_\e,k_\e)$ compares to the Lebesgue measure of~$\omega_\e$.

\begin{theorem}\label{thm3}
Let $\Omega$ be a smooth bounded domain of $\R^d$ whose boundary is divided into two disjoint open subsets $\Gamma_N$ and $\Gamma_D$, and let $K_N$ be a fixed compact subset of $\Gamma_N$. Then, the following facts hold true: 
\begin{itemize}
\item[(i)]
There exists a constant $C >0 $, depending on $\Omega$, $\Gamma_N$, $\Gamma_D$ and the compact set $K_N$ such that
\begin{equation} \label{est_cRN_cap}
\forall \;k > 0, \;\; \forall\; \omega \subset K_N,\quad
\capar^N(\omega,k) \:\: \leq \:\: C\, \capa(\omega).
\end{equation}
\item[(ii)]
Let $M>0$ be a fixed constant. Then, there exists $C> 0$, depending on $\Omega$ and $M$, such that
\begin{equation} \label{est_cRN_vol}
\forall 0 \leq k \leq M, \:\: \forall\; \omega \Subset \Gamma_N,\quad
C k \, |\omega| \;\leq\; \capar^N(\omega,k)
\;\leq\; k \, |\omega|.
\end{equation}
\item[(iii)] Let $\omega \Subset \Gamma_N$ be a fixed subset and let $k_\e$ be a sequence of
positive numbers such that $k_\e \to \infty$. Then
\begin{equation}\label{lim_cRN_cap}
\frac{1}{C} \: \leq \: \frac{ \capar^N(\omega,k_\e)}{ \capa(\omega)} \: \leq  \: C, 
\end{equation}
for a certain constant $C$ which depends on $\omega$.
\end{itemize}
\end{theorem}
\medskip

The next result, which is proved in \cref{sec_4} under additional assumptions on the sets $\omega_\e$, 
improves point (iii) in \cref{thm3}: it shows that the asymptotic expression \cref{asymp_uN} is formally consistent with that \cref{asymp_pertN} obtained in the case of a Dirichlet perturbation of $\Gamma_N$ when $k_\e$ tends to $\infty$ sufficiently fast.

\begin{theorem} \label{thm5}
Under the assumptions of \cref{thm3}, let $\omega_\e$ be a sequence of subsets of $\Gamma_N$ satisfying \cref{eq.assumsepKN}, and such that, additionally:
\begin{itemize}
\item If $d=2$, the set $\omega_\e$ is an arc of length $\e$;
\item If $d=3$, the set $\omega_\e$ is the image $\Phi(B_\e)$ of the planar disk 
$$\D_\e := \Big\{x = (x_1,x_2,0) \in \R^3, \quad x_1^2 + x_2^2 < \e^2 \Big\},$$
by a fixed, smooth diffeomorphism $\Phi~: \R^3 \rightarrow \R^3$.
\end{itemize}

Assume that $\capa(\omega_\e) \to 0$ and that 
$k_\e$ tends to infinity faster than $\lvert \omega_\e\lvert^3$ goes to $0$, 
i.e. $k_\e \lvert \omega_\e\lvert^3  \to \infty$ as $\e \to 0$.
Then,
$$C \: \leq \: \frac{\capar^N(\omega_\e,k_\e) }{\capa(\omega_\e)} \leq 1,$$
for a certain constant $C>0$ which is independent of $\e$.
\end{theorem}

This result expresses that the expansion \cref{asymp_uN} of $u_\e - u_0$ has the formal limit \cref{asymp_pertN} when $k_\e$ tends to $\infty$ sufficiently fast. This aligns with the intuition that the Robin boundary condition
$\gamma\frac{\partial u_\e}{\partial n} + k_\e u_\e = 0$ on $\omega_\e$ ``degenerates'' to a Dirichlet boundary condition
$u_\e = 0$ when $k_\e \to \infty$ sufficiently fast. Note that the particular form
of the sets $\omega_\e$ in this section implies that $|\omega_\e| = \o\big(\capa(\omega_\e)\big)$.

\begin{remark}\mbox{ }
\begin{enumerate}[(i)]
\item
Under the assumptions on the shapes of the $\omega_\e$ in \cref{thm5}, the measure~$\mu$ that appears
in the expansion \cref{asymp_pertN} corresponding to a ``pure'' Dirichlet perturbation of the Neumann region can be fully identified, see~\cite{bonnetier_dapogny_vogelius}. It would be interesting to conduct an analogous study in the present context of a Robin perturbation: an explicit formula for the first-order expansion of $u_\e^{N,k_\e}$ which holds true in a uniform fashion with respect to the parameter $k_\e$ would explain the smooth transition between the rates $k_\e \lvert \omega_\e\lvert$ and $\capa(\omega_\e)$ in \cref{est_cRN_vol} and \cref{est_cRN_cap} when $k_\e$ passes from $0$ to $\infty$. We refer to e.g. \cite{dapogny2017uniform,nguyen2009representation} about such uniform expansions in the context of volumic perturbations of the coefficients of an elliptic equation. In particular, we conjecture that the rate at which $k_\e$ should tend to $\infty$ exhibited by \cref{thm5} is not optimal. This would be elucidated by such an explicit expansion. 
\item In~\cite{bonnetier_brito_dapogny_estevez}, it is shown that
$u_\e - u_0$ scales like~$|\omega_\e|$, when one replaces a 
homogeneous Neumann boundary condition by an inhomogeneous one on $\omega \Subset \Gamma_N$,
which is consistent with \cref{est_cRN_vol} for the Robin condition $\gamma\frac{\partial u}{\partial n} + k u$
with ``small'' $k$.
\end{enumerate}
\end{remark}
\par\medskip 

We conclude this section with a result about the monotonicity of the Robin-Neumann capacity $\capar^{N}(\omega,k)$ with respect to the support $\omega$ and the parameter $k$ of the Robin boundary condition.

\begin{proposition}\label{prop.monotonecaparN}
Let $\Omega$ be a smooth bounded domain in $\R^d$, whose boundary is divided into two disjoint open subsets $\Gamma_N$ and $\Gamma_D$. 
\begin{enumerate}[(i)]
\item Let $\omega$ be a Lipschitz open subset of the Neumann region $\Gamma_N$ and let $0 < k_1 \leq k_2$ be positive real numbers. It holds: 
$$ \capar^{N}(\omega,k_1) \leq \capar^N(\omega,k_2).$$
\item Let $k$ be a fixed real number and let $\omega_1 \subset \omega_2$ be two Lipschitz open subsets of $\Gamma_N$. It holds: 
$$ \capar^{N}(\omega_1,k) \leq \capar^N(\omega_2,k).$$
\end{enumerate}
\end{proposition} \par\medskip


\subsection{Robin perturbation of the Dirichlet boundary condition on $\Gamma_D$}\label{subsec.stateRobDir}

\noindent 
In this section, we consider the situation where $\omega_\e \Subset \Gamma_D$, 
and where the perturbed potential $u_\e^{D,k_\e}$ results from the perturbation of the homogeneous Dirichlet condition on $\omega_\e$ by a Robin condition with parameter $k_\e$, see \cref{eq_pertD}. In this situation,
the asymptotic behavior of $u_\e^{D,k_\e}$ is described by the following result:

\begin{theorem}\label{thm2}
Let $\Omega$ be a smooth bounded domain of $\R^d$ whose boundary is divided into two disjoint open subsets $\Gamma_N$ and $\Gamma_D$, and let $K_D \subset \Gamma_D$ be a fixed compact set, see \cref{sec.statepb}.
Let $\omega_\e$ be a sequence of subsets of $\Gamma_D$ that satisfies
the separation assumption~\cref{eq.assumsepKD}
and let $k_\e$ be any sequence of positive real numbers. 
Assume that the Robin-Dirichlet capacity $\capar^D(\omega_\e,k_\e) 
$ tends to $0$ as $\e \to 0$.
Then, there exists a subsequence, still denoted by $\e$, 
as well as a positive Radon measure $\mu$ with unit mass, supported in $\Gamma_D$, 
such that the following asymptotic expansion holds true:
\begin{equation} \label{asymp_uD}
\text{For } x \in \Omega, \quad u_\e^{D,k_\e}(x) = u_0(x)
+ \capar^D(\omega_\e,k_\e) \, 
\left\langle\mu, \gamma \frac{\partial u_0}{\partial n}\frac{\partial G}{\partial n_y}(x,\cdot)\right\rangle
\;+\; \o\left(\capar^D(\omega_\e,k_\e)\right).
\end{equation}
\end{theorem}
\medskip

Alongside \cref{thm2}, we show the following result.

\begin{theorem}\label{thm4}
Let $\Omega$ be a smooth bounded domain of $\R^d$ whose boundary is divided into two disjoint open subsets $\Gamma_N$ and $\Gamma_D$, and let $K_D$ be a fixed compact subset of $\Gamma_D$. Then, the following facts hold true: 
\begin{itemize}
\item[(i)]
For any $k > 0$, and any $\omega \subset K_D$, there exists $C > 0$ such that:
\begin{equation} \label{est_cRD_e}
\capar^D(\omega,k) \:\: \leq\:\: C\, e(\omega).
\end{equation}
The constant $C$ depends on $\Omega$, $\Gamma_N$, $\Gamma_D$ and $K_D$, but it is independent of $\omega$ and $k$. 
\item[(ii)]
For any $\omega \Subset \Gamma_D$, and for any $k > 0$, it holds: 
\begin{equation} \label{est_cRD_vol}
\capar^D(\omega,k) \:\: \leq \:\:  \ds\frac{1}{k} |\omega|.
\end{equation}

\item[(iii)]
Let $\omega_\e$ denote a sequence of subsets of a fixed compact set $K_D \subset \Gamma_D$ and let $k_\e > 0$ be a sequence of positive numbers. Assume that there exists $M >0$ such that $k_\e \leq M$.
Then
\begin{equation} \label{lim_cRD_cap}
\frac{1}{C} \: \leq \: \frac{\capar^D(\omega_\e,k_\e)}{e(\omega_\e)} \: \leq \: C,
\end{equation}
for a constant $C >0$ depending on $\Omega$, $\Gamma_N$, $\Gamma_D$, $K_D$ and the constant $M$ but not on $\e$.
\end{itemize}
\end{theorem}

\begin{remark}
The estimate~\cref{lim_cRD_cap} shows that when $k_\e \to 0$, $\capar^D(\omega_\e,k_\e)$ scales like
the Neumann capacity $e(\omega_\e)$, consistently with \cref{asymp_pertD}. In other words, 
the Robin boundary condition on $\omega_\e$ has an effect similar here to that of a Neumann condition in this regime.
\end{remark}

Eventually, the Robin-Dirichlet capacity $\capar^{D}(\omega,k)$ satisfies monotonicity properties with respect to the set $\omega$ and the parameter $k$ of the Robin boundary condition, akin to those stated in \cref{prop.monotonecaparN} in the case of the Robin-Neumann capacity.

\begin{proposition}\label{prop.monotonecaparD}
Let $\Omega$ be a smooth bounded domain in $\R^d$, whose boundary is divided into two disjoint open subsets $\Gamma_N$ and $\Gamma_D$. 
\begin{enumerate}[(i)]
\item Let $\omega$ be a Lipschitz open subset of the Dirichlet region $\Gamma_D$ and let $0 < k_1 \leq k_2$ be positive real numbers. It holds: 
$$ \capar^{D}(\omega,k_2) \leq \capar^D(\omega,k_1).$$
\item Let $k$ be a fixed real number and let $\omega_1 \subset \omega_2$ be two Lipschitz open subsets of $\Gamma_D$. It holds: 
$$ \capar^{D}(\omega_1,k) \leq \capar^D(\omega_2,k).$$
\end{enumerate}
\end{proposition}

\section{Proofs of the results about Robin perturbation of the Neumann boundary $\Gamma_N$} \label{sec_3}

\noindent This section takes place in the situation of \cref{subsec.stateRobNeu}: $\omega_\e$ is a sequence 
of subsets of $\Gamma_N$ which is well-separated from $\Gamma_D$ in the sense of \cref{eq.assumsepKN}, and the perturbed version of the background problem \cref{eq.conducbg2} under scrutiny is that \cref{eq_pertN} where the Neumann boundary condition on $\omega_\e$ is replaced by a Robin boundary condition with parameter $k_\e$. 
We aim to prove \cref{thm1,thm3} and \cref{prop.monotonecaparN}. 

For notational brevity, throughout this section, we use the shorthands $\chi_\e \equiv \chi(\omega_\e,k_\e)$ and $u_\e= u_\e^{N,k_\e}$ for the equilibrium distribution and the perturbed potential, solutions to \cref{def_chi,eq_pertN}, respectively.

\subsection{Preliminary estimates}

\noindent Let us first recall the following basic version of the Poincar\'e inequality for functions in $H^1_{\Gamma_D}(\Omega)$: there exists a constant $C >0$, which depends only on $\Omega$, $\Gamma_D$ and the bounds $\gamma_-,\gamma_+$ on the conductivity $\gamma$ in \cref{cond_ell}, such that:
\begin{equation}\label{eq.Poincare}
\forall v \in H^1_{\Gamma_D}(\Omega), \quad \frac{1}{C} \lvert\lvert v \lvert\lvert_{H^1(\Omega)}^2 \:\: \leq \:\:  \int_\Omega \gamma\lvert\nabla v \lvert^2 \:\d x  \:\: \leq \:\:  \lvert\lvert v \lvert\lvert_{H^1(\Omega)}^2.
\end{equation}
Given $g \in \calC^\infty(\R^d)$, we now define $\varphi_\e$ as the $H^1(\Omega)$ solution to the following boundary-value problem:
\begin{equation}\label{def_vphie}
\left\{
\begin{array}{cccl}
-\dv(\gamma\nabla \varphi_\e) &=& 0 & \text{in } \Omega, \\[5pt]
\varphi_\e &=& 0 & \text{on } \Gamma_D, \\[5pt]
\gamma \frac{\partial \varphi_\e}{\partial n} &=& 0 & \text{on } \Gamma_N \setminus \overline{\omega_\e}, \\[5pt]
\gamma \frac{\partial \varphi_\e}{\partial n}  + k_\e \varphi_\e &=& k_\e g & \text{on } \omega_\e.
\end{array}
\right. 
\end{equation}
Equivalently, the variational form of this problem is
\begin{equation}\label{eq.varfcapfuncrobgen}
\text{Find } \varphi_\e \in H^1_{\Gamma_D}(\Omega) 
\text{ s.t. } \forall v \in H^1_{\Gamma_D}(\Omega), \quad 
\int_\Omega \gamma\nabla \varphi_\e \cdot \nabla v \:\d x 
+k_\e\int_{\omega_\e} \varphi_\e v \:\d s = k_\e \int_{\omega_\e}  gv \:\d s.
\end{equation}
Note that when $g = 1$, $\varphi_\e$ is the solution $\chi_\e$ to \cref{def_chi} when $\omega = \omega_\e$ and $k = k_\e$.
\bigskip

We next estimate the function $\varphi_\e$.
\begin{proposition}\label{prop.estvphie} 
Let $\omega_\e$ be a sequence of subsets of $\Gamma_N$ satisfying 
the separation assumption~\cref{eq.assumsepKN} and let $k_\e$ be any sequence 
of positive real numbers; let~$g\in \calC^\infty(\R^d)$, and let~$\varphi_\e$ 
be the unique solution to~\cref{def_vphie}. 
Then,
\begin{enumerate}[(i)]
\item 
The following pointwise bounds hold true:
\begin{equation}\label{eq.Linftyvphie}
 \min \left(0, \inf\limits_{x\in \overline{\Omega}} g(x) \right)  \: \leq \:  \varphi_\e(x) \: \leq \: \max\left(0,\sup\limits_{x\in \overline{\Omega}} g(x)\right) \text{ for a.e. } x\in \overline\Omega.
 \end{equation} 
\item 
The following $H^1$ estimate holds true:
\begin{equation}\label{eq.H1vphie}
\int_\Omega \gamma \lvert \nabla \varphi_\e \lvert^2 \:\d x  
\: \leq \:  
C \lvert\lvert g \lvert\lvert_{\calC^1(\R^d)}^2 \capar^N(\omega_\e,k_\e),
\end{equation}
where the constant $C$ depends on $\Omega$, $\Gamma_N$, $\Gamma_D$, the compact set $K_N$, but not on $\e$.
\item 
The following improved $L^2$ estimate holds true: 
\begin{equation}\label{eq.L2vphie}
\int_\Omega  \varphi_\e^2 \:\d x  
\: \leq \:  
C \lvert\lvert g \lvert\lvert_{\calC^1(\R^d)}^2 \capar^N(\omega_\e,k_\e)^{\frac{3}{2}}.
\end{equation}
Again, $C$ depends on $\Omega$, $\Gamma_N$, $\Gamma_D$, $K_N$, but not on $\e$.
\end{enumerate}
\end{proposition}

\begin{proof}
\noindent \textit{Proof of (i).} This result is a variant of the ``classical'' maximum principle applied to the present boundary-value problem \cref{def_vphie}. Its proof is inspired by the arguments in e.g \S 9.7 in \cite{brezis2010functional}. For notational brevity, let $M$ denote the quantity in the right-hand side of \cref{eq.Linftyvphie}, and let us set $\psi_\e(x) := \max(\varphi_\e(x)- M,0)$. Since the mapping $\R \ni t \mapsto \max(t,0)$ is Lipschitz continuous, the function $\psi_\e$ belongs to $H^1_{\Gamma_D}(\Omega)$ and its derivative reads:
$$ \nabla \psi_\e = \mathds{1}_{\varphi_\e \geq M} \nabla \varphi_\e,$$ 
see e.g. \S 3.1.2 in \cite{henrot2018shape}. Insertion of $\psi_\e$ as test function into \cref{eq.varfcapfuncrobgen} leads to:
$$ \int_\Omega \gamma \lvert\nabla \varphi_\e \lvert^2 \mathds{1}_{\left\{\varphi_\e \geq  M\right\}} \:\d x +  \int_{\omega_\e} k_\e  \varphi_\e (\varphi_\e - M) \mathds{1}_{\left\{\varphi_\e \geq  M\right\}}\:\d s =  \int_{\omega_\e} k_\e g (\varphi_\e-M) \mathds{1}_{\left\{\varphi_\e \geq  M\right\}}\:\d s.$$
By re-arranging this equality, we obtain:
$$ \int_\Omega \gamma \lvert\nabla \varphi_\e \lvert^2 \mathds{1}_{\left\{\varphi_\e \geq  M\right\}} \:\d x =  \int_{\omega_\e} k_\e (g  - \varphi_\e)(\varphi_\e-M) \mathds{1}_{\left\{\varphi_\e \geq  M\right\}}\:\d s \leq 0,$$
where the inequality is a consequence of the definition of $M$.
It follows that $\psi_\e (x) = 0$ a.e. and so $ \varphi_\e(x) \leq M$ for a.e. $x \in \overline\Omega$, as desired. The left-hand inequality in \cref{eq.Linftyvphie} is proved by similar means.\par
\bigskip

\noindent \textit{Proof of (ii).} Taking $v = \varphi_\e$ as test function in \cref{eq.varfcapfuncrobgen}, we obtain the following a priori estimate for $\varphi_\e$: 
$$ \int_\Omega \gamma \lvert \nabla \varphi_\e \lvert^2 \:\d x = k_\e \int_{\omega_\e} (g-\varphi_\e) \varphi_\e \:\d s.$$
Introducing the function $\chi_\e $ in \cref{def_chi} into the above right-hand side, it follows:
\begin{eqnarray*}
\ds\int_\Omega  \gamma|\nabla \varphi_\e |^2  \:\d x
&=&
\ds\int_{\omega_\e} k_\e(g - \varphi_\e) \varphi_\e \: \d s
\\
&=&
\ds\int_{\omega_\e} \left(\gamma \frac{\partial \chi_\e}{\partial n} + k_\e \chi_\e\right) (g - \varphi_\e)\varphi_\e \: \d s
\\
&=&
\ds\int_{\omega_\e} \gamma \frac{\partial \chi_\e}{\partial n} \, g \varphi_\e \:\d s
\;-\; \ds\int_{\omega_\e} \gamma \frac{\partial \chi_\e}{\partial n} \, \varphi_\e^2 \:\d s
\;-\; \ds\int_{\omega_\e} k_\e \chi_\e (g - \varphi_\e)^2 \:\d s
\;+\; \ds\int_{\omega_\e} k_\e \chi_\e (g - \varphi_\e) g\:\d s
\\
&\leq&
\ds\int_{\partial \Omega} \gamma \frac{\partial \chi_\e}{\partial n} \, g \varphi_\e \:\d s
+ \ds\int_{\partial \Omega} \gamma \frac{\partial \varphi_\e}{\partial n} \, g \chi_\e \:\d s,
\end{eqnarray*}
where we have used the Robin boundary condition for $\chi_\e$ in \cref{def_chi} to pass 
from the first line to the second one and the Robin boundary condition for $\varphi_\e$ in \cref{def_vphie} to get the third one. We have also used the fact that $0 \leq \chi_\e \leq 1$ 
as a consequence of the pointwise bounds \cref{eq.Linftyvphie},
and the fact that the normal derivative $\gamma \frac{\partial \chi_\e}{\partial n}$ is therefore nonnegative on $\omega_\e$.
Now using the Green's formula, the Cauchy-Schwarz inequality and the Poincar\'e's inequality \cref{eq.Poincare}, we obtain:
\begin{equation*}
 \int_\Omega \gamma \lvert \nabla \varphi_\e \lvert^2 \:\d x \:\:\leq\:\:  C \lvert\lvert g \lvert\lvert_{\calC^1(\overline{\Omega})} \lvert\lvert \nabla \varphi_\e \lvert\lvert_{L^2(\Omega)^d} \lvert\lvert \nabla \chi_\e \lvert\lvert_{L^2(\Omega)^d},
\end{equation*}
and so, after cancellation of $ \lvert\lvert \nabla \varphi_\e \lvert\lvert_{L^2(\Omega)^d} $ 
on both sides and squaring:
$$ \int_\Omega \gamma \lvert \nabla \varphi_\e \lvert^2 \:\d x \:\: \leq \:\:  C \lvert\lvert g \lvert\lvert_{\calC^1(\overline{\Omega})} ^2  \int_\Omega \gamma \lvert \nabla \chi_\e \lvert^2 \:\d x. $$
Now, a simple calculation yields:
\begin{equation*}
\begin{array}{>{\displaystyle}cc>{\displaystyle}l}
 \int_\Omega \gamma \lvert \nabla \chi_\e \lvert^2 \:\d x &=& k_\e \int_{\omega_\e} (1-\chi_\e) \chi_\e \:\d s \\[1em]
 &\leq&  k_\e \int_{\omega_\e} (1-\chi_\e)  \:\d s\\[1em]
 & =& \capar^N(\omega_\e,k_\e),
 \end{array}
\end{equation*}
as follows from another application of the pointwise bounds \cref{eq.Linftyvphie} and the definition \cref{eq.defcapar} of $\capar^{N}(\omega_\e,k_\e)$. 
This terminates the proof of \cref{eq.H1vphie}.
\par\medskip

\noindent \textit{Proof of (iii).} We rely on the so-called Aubin-Nitsche duality trick, see \cite{aubin1967behavior,nitsche1968kriterium} or \cite{ciarlet2002finite} in the context of the Finite Element Method.
Let $p_\e \in H^1_{\Gamma_D}(\Omega)$ be the unique solution to the following boundary value problem:
\begin{equation}\label{eq.pbpe}
\left\{
\begin{array}{cl}
-\dv(\gamma\nabla p_\e) = -\varphi_\e & \text{in } \Omega, \\
p_\e = 0 & \text{on } \Gamma_D,\\
\gamma \frac{\partial p_\e}{\partial n} = 0 & \text{on } \Gamma_N.
\end{array}
\right.
\end{equation}
We introduce two open sets $\calO_1 \Subset \calO_2$ such that $K_N \subset \calO_1$ and $\overline{\calO_2} \cap \overline{\Gamma_D} = \emptyset$, 
as well as a smooth cut-off function $\eta$ satisfying $\eta \equiv 1$ on $\calO_1$ and $\eta\equiv 0$ on $\R^d \setminus \calO_2$.  
We now use classical elliptic regularity theory \cite{brezis2010functional,gilbarg2015elliptic} and the Sobolev embedding theorem \cite{adams2003sobolev} to see that 
that $p_\e \in H^3(\calO_2 \cap \Omega)$ with
\begin{equation}\label{eq.regpeAN} 
\lvert\lvert p_\e \lvert\lvert_{\calC^1(\calO_2 \cap \Omega)} \: \leq \: C \lvert\lvert p_\e \lvert\lvert_{H^3(\calO_2 \cap \Omega)} \: \leq \: C \lvert\lvert \varphi_\e \lvert\lvert_{H^1(\Omega)}.
\end{equation}
It follows that: 
$$\begin{array}{>{\displaystyle}cc>{\displaystyle}l}
\int_\Omega \varphi_\e^2 \:\d x &=& -\int_\Omega \gamma \nabla p_\e \cdot \nabla \varphi_\e \:\d x  \\[1em]
&=& k_\e \int_{\omega_\e} (g - \varphi_\e) p_\e \:\d s,
\end{array}
$$
where we have used \cref{eq.pbpe} to obtain the second line.  
As in the proof of (ii), introducing the equilibrium distribution $\chi_\e$ via its Robin boundary condition on $\omega_\e$, we obtain:
$$\begin{array}{>{\displaystyle}cc>{\displaystyle}l}
\int_\Omega \varphi_\e^2 \:\d x &=& \int_{\omega_\e} \gamma \frac{\partial \chi_\e}{\partial n}(g - \varphi_\e) p_\e \:\d s + \int_{\omega_\e} k_\e \chi_\e (g - \varphi_\e) p_\e \:\d s \\[1em]
&=& \int_{\omega_\e} \gamma \frac{\partial \chi_\e}{\partial n}(g - \varphi_\e) p_\e \:\d s + \int_{\omega_\e} \gamma \frac{\partial \varphi_\e}{\partial n}\chi_\e  p_\e \:\d s \\[1em]
&=& \int_{\omega_\e} \gamma \frac{\partial \chi_\e}{\partial n}(g - \varphi_\e) p_\e \:\d s + \int_{\Omega} \gamma \nabla \varphi_\e \cdot \nabla (\chi_\e p_\e \eta)  \:\d x.
\end{array}$$
Now, the non negativity of $\frac{\partial \chi_\e}{\partial n}$ and the definition \cref{eq.defcapar} of the Robin-Neumann capacity together imply that: 

$$\begin{array}{>{\displaystyle}cc>{\displaystyle}l}
\int_\Omega \varphi_\e^2 \:\d x  &\leq& \lvert\lvert p_\varepsilon \lvert\lvert_{\calC(\omega_\e)} \int_{\omega_\e} \gamma \frac{\partial \chi_\e}{\partial n} \lvert g - \varphi_\e\lvert  \:\d s +  \int_{\Omega} \gamma \nabla \varphi_\e \cdot \nabla (\chi_\e p_\e \eta)  \:\d x  \\[1em]
&\leq& \lvert\lvert p_\varepsilon \lvert\lvert_{\calC(\omega_\e)} \lvert\lvert g \lvert\lvert_{\calC^1(\partial \Omega)}  \int_{\omega_\e} \gamma \frac{\partial \chi_\e}{\partial n}  \:\d s + \lvert\lvert p_\varepsilon \lvert\lvert_{\calC^1(\calO_2 \cap \overline\Omega)} \lvert\lvert \nabla \varphi_\e \lvert\lvert_{L^2(\Omega)^d} \lvert\lvert \nabla \chi_\e \lvert\lvert_{L^2(\Omega)^d}  \\[1em]
&\leq& C  \lvert\lvert p_\varepsilon \lvert\lvert_{\calC^1(\calO_2 \cap \overline\Omega)} \lvert\lvert g \lvert\lvert_{\calC^1(\partial \Omega)} \capar^N(\omega_\e,k_\e),
\end{array}
$$
where we have used the $H^1$ estimate \cref{eq.H1vphie}. The desired result follows from \cref{eq.regpeAN}.
\end{proof}\par\medskip


\subsection{The representation formula}

\noindent The aim of this section is to prove the representation formula \cref{asymp_uN} 
for the perturbed potential $u_\e$, solution to \cref{eq_pertN}. The argument is based on the idea of compensated compactness \cite{murat2018h}. It is quite reminiscent of that used on the proof of the representation formula \cref{asympYM} for volumic  perturbations of the conductivity considered in \cite{capdeboscq2003general}, 
see also \cite{bonnetier_dapogny_vogelius}. 
 
\begin{proof}[Proof of \cref{thm1}]
Let $r_\e := u_\e - u_0$ denote the error between the perturbed and background potentials. 
This function is the unique $H^1(\Omega)$ solution to the problem:
\begin{equation}\label{eq.pbre}
\left\{
\begin{array}{cl}
-\dv(\gamma\nabla r_\e) = 0 & \text{in } \Omega, \\[5pt]
r_\e = 0 & \text{on } \Gamma_D, \\[5pt]
\gamma \frac{\partial r_\e}{\partial n} = 0 & \text{on } \Gamma_N \setminus \overline{\omega_\e}, \\[5pt]
\gamma \frac{\partial r_\e}{\partial n} + k_\e r_\e= -k_\e u_0 & \text{on } \omega_\e.
\end{array}
\right.
\end{equation}
Let $x \in \Omega$ be a fixed point. Using the definition \cref{eq_Green} of the Green's function $G(x,y)$ to represent the value $r_\e(x)$ of the error at $x$, we obtain, after integration by parts: 
$$
\begin{array}{>{\displaystyle}cc>{\displaystyle}l}
r_\e(x) &=& - \int_\Omega \dv_y(\gamma(y) \nabla_y G(x,y) ) \: r_\e(y) \:\d y
\\[1em]
&=& 
- \int_{\partial \Omega} \gamma(y) \frac{\partial G}{\partial n_y} (x,y)r_\e(y) \:\d s(y) 
+ \int_\Omega \gamma(y) \nabla_y G(x,y) \cdot \nabla r_\e(y) \:\d y
\\[1em]
&=&  
\int_\Omega \gamma(y) \nabla_y G(x,y) \cdot \nabla r_\e(y) \:\d y,
\end{array}
$$
where we have used the fact that $\gamma(y) \frac{\partial G}{\partial n_y}(x,y)$ vanishes 
for~$y\in \Gamma_N$ and that~$r_\e$ vanishes on~$\Gamma_D$. 
Another integration by parts yields:
$$
\begin{array}{>{\displaystyle}cc>{\displaystyle}l}
r_\e(x) 
&=& - \int_\Omega \dv(\gamma \nabla r_\e) (y) \: G(x,y) \:\d y 
+ \int_{\partial \Omega} \gamma(y)\frac{\partial r_\e}{\partial n}(y) G(x,y) \:\d s(y) 
\\[1em]
&=&  
-k_\e \int_{\omega_\e} r_\e(y) G(x,y) \:\d s(y)  -k_\e \int_{\omega_\e} u_0(y) G(x,y) \:\d s(y). 
\end{array}
$$
Now introducing the Robin equilibrium potential $\chi_\e$ defined in~\cref{def_chi},
we obtain:
$$ 
r_\e(x) =  
- \int_{\omega_\e} \gamma \frac{\partial \chi_\e}{\partial n}(y) r_\e(y) G(x,y) \:\d s(y) 
- k_\e \int_{\omega_\e} \chi_\e(y) r_\e(y) G(x,y) \:\d s(y) 
-k_\e \int_{\omega_\e} u_0(y) G(x,y) \:\d s(y).
$$
Let now $\phi$ be a smooth function which coincides with $G(x,\cdot)$ 
on a fixed neighborhood of the compact set $K_N$ containing the vanishing subsets $\omega_\e$. 
Another use of Green's formula yields: 
$$ 
\begin{array}{>{\displaystyle}cc>{\displaystyle}l}
r_\e(x) 
&=&  
- \int_{\Omega} \gamma \nabla \chi_\e  \cdot \nabla (r_\e \phi) \:\d y 
- k_\e \int_{\omega_\e} \chi_\e  r_\e \phi \:\d s -k_\e \int_{\omega_\e} u_0 \phi \:\d s 
\\[1em]
&=&  
- \int_{\Omega} (\gamma \nabla \chi_\e  \cdot \nabla r_\e) \phi \:\d y 
- k_\e \int_{\omega_\e} \chi_\e  r_\e \phi \:\d s -k_\e \int_{\omega_\e} u_0 \phi \:\d s  +R_{1,\e}  
\\[1em]
&=&  
- \int_{\Omega} \gamma \nabla (\chi_\e  \phi) \cdot \nabla r_\e \:\d y 
- k_\e \int_{\omega_\e} \chi_\e  r_\e \phi \:\d s -k_\e \int_{\omega_\e} u_0 \phi \:\d s  + R_{1,\e} + R_{2,\e},
\end{array}
$$
where we have introduced the remainders
$$ 
R_{1,\e} :=  - \int_{\Omega} (\gamma \nabla \chi_\e  \cdot \nabla \phi) r_\e \:\d y, 
\text{ and } 
R_{2,\e} := \int_{\Omega} (\gamma \nabla r_\e \cdot \nabla \phi) \chi_\e  \:\d y. 
$$
Yet another use of Green's formula reveals that
$$ 
\begin{array}{>{\displaystyle}cc>{\displaystyle}l}
r_\e(x) 
&=&  
- \int_{\partial \Omega} \gamma \frac{\partial r_\e}{\partial n} \chi_\e  \phi \:\d s 
- k_\e \int_{\omega_\e} \chi_\e  r_\e \phi \:\d s -k_\e \int_{\omega_\e} u_0 \phi \:\d s  
+ R_{1,\e} + R_{2,\e} 
\\[1em]
&=&  
- \int_{\omega_\e} \gamma \frac{\partial r_\e}{\partial n} \chi_\e  \phi \:\d s 
- k_\e \int_{\omega_\e} \chi_\e  r_\e \phi \:\d s -k_\e \int_{\omega_\e} u_0 \phi \:\d s  
+ R_{1,\e} + R_{2,\e} 
\\[1em]
&=&  
k_\e \int_{\omega_\e} (\chi_\e  -1) u_0 \phi \:\d s  + R_{1,\e} + R_{2,\e},
\end{array} 
$$
Let us now write this equality under the form:
\begin{equation}\label{eq.rebeforecvGN}
r_\e(x) = \capar^N(\omega_\e,k_\e)   
\int_{\omega_\e} \frac{1}{ \capar^N(\omega_\e,k_\e) } k_\e (\chi_\e  -1) u_0 \phi \:\d s 
+ R_{1,\e} + R_{2,\e}.
\end{equation}
Now, the sequence $\frac{1}{ \capar^N(\omega_\e,k_\e) } k_\e (\chi_\e  -1)$ is bounded 
in $L^1(\partial \Omega)$, as an immediate consequence of the definition~\cref{eq.defcapar} of the Neumann-Robin capacity.
Hence, there exists a subsequence of the $\e$ (still denoted by $\e$) and 
a positive Radon measure~$\mu$ with unit mass, supported on the fixed compact subset $K_N$ (in particular, on $\Gamma_N$), such that 
$$
\forall \varphi \in \calC(\partial\Omega), \quad   
\int_{\omega_\e} \frac{1}{ \capar^N(\omega_\e,k_\e) } k_\e (\chi_\e  -1) \varphi \:\d s 
\xrightarrow{\e\to 0} 
- \int_{\partial \Omega} \varphi \:\d\mu.
$$
Note that the sign of $\mu$ stems from the fact that $\chi_\e-1 \leq 0$, as a consequence of the pointwise bounds \cref{eq.Linftyvphie}.  
Moreover, using the~$H^1$ and improved~$L^2$ estimates in~\cref{prop.estvphie}, we see that
$$ 
R_{1,\e} + R_{2,\e}  = \o\left(\capar^N(\omega_\e,k_\e)  \right).
$$
Combining this estimate with \cref{eq.rebeforecvGN}, we arrive at the desired formula \cref{asymp_uN}.
\end{proof}\par\medskip


\subsection{Comparisons between scalings: proof of \cref{thm3}} \label{s3.3}

\noindent \textit{Proof of~(i):}
Let $\omega \Subset \Gamma_N$ and $k >0$ be given; we use the shorthands $\chi = \chi(\omega,k)$ and $\chi^\infty=  \chi^\infty(\omega)$ for the Robin and Dirichlet equilibrium distributions defined by \cref{def_chi} and \cref{def_chi_inf}. Let $r$ be the difference
$r = \chi - \chi^\infty$, which satisfies the following boundary-value problem:
\[ 
\left\{ \begin{array}{ccll}
- \textrm{div}( \gamma \nabla r ) &=& 0 & \textrm{in}\; \Omega,
\\[5pt]
r &=& 0 & \textrm{on}\; \Gamma_D,
\\[5pt]
\gamma \frac{\partial r}{\partial n} &=& 0 & \textrm{on}\; \Gamma_N \setminus \overline\omega,
\\[5pt]
\gamma  \frac{\partial r}{\partial n} + k r &=& -  \gamma \frac{\partial \chi^\infty}{\partial n}  & \text{on } \omega.
\end{array} \right.
\]
The corresponding variational problem reads:
\begin{eqnarray*}
\forall  v \in H^1_{\Gamma_D}(\Omega), \quad
\ds\int_\Omega \gamma \nabla r \cdot \nabla v \:\d x \;+\; k \ds\int_{\omega} r v \:\d s
&=& - \int_{\omega} \gamma \frac{\partial \chi^\infty}{\partial n} v \:\d s.
\end{eqnarray*}
On the one hand, by taking $v = r$ as a test function in this identity and using the fact that $\chi^\infty$ 
solves \cref{def_chi_inf}, we see that:
\begin{eqnarray*}
\ds\int_{\Omega} \gamma | \nabla r|^2 \:\d x  \;+\; k \ds\int_{\omega} r^2 \:\d s
&=&
-\ds\int_\omega \gamma \frac{\partial \chi^\infty}{\partial n} r  \:\d s \\
&=& -\ds\int_{\partial \Omega} \gamma \frac{\partial \chi^\infty}{\partial n} r \:\d s
\\
&=&
- \ds\int_\Omega \gamma \nabla \chi^\infty \cdot \nabla r \:\d x
\\
&\leq&
C\, ||\nabla r||_{L^2(\Omega)^d} \, ||\nabla \chi^\infty||_{L^2(\Omega)^d},
\end{eqnarray*}
so that in particular:
\begin{equation} \label{est_r}
||\nabla r||_{L^2(\Omega)^d}  \:\: \leq \:\:  C\, ||\nabla \chi^\infty||_{L^2(\Omega)^d}.
\end{equation}
On the other hand, elementary manipulations of the definitions \cref{def_chi,def_chi_inf} of $\chi^\infty$ 
and $\chi$ yield: 
\begin{eqnarray*}
\capar^N(\omega,k) &=& \ds\int_{\omega} \gamma \frac{ \partial \chi}{\partial n} \:\d s \\
&=& \ds\int_{\omega} \gamma  \frac{ \partial \chi}{\partial n} \, \chi^{\infty} \:\d s
\\
&=&
\ds\int_{\partial \Omega} \gamma  \frac{ \partial \chi}{\partial n} \, \chi^{\infty} \:\d s.
\end{eqnarray*}
Now, by integration by parts, we see that:
\begin{eqnarray*}
\capar^N(\omega,k)  &=& \ds\int_\Omega  \gamma \nabla \chi \cdot \nabla \chi^\infty \:\d x.
\\
&=& 
\ds\int_\Omega \gamma \nabla r \cdot \nabla \chi^\infty \:\d x
\;+\; 
\ds\int_\Omega \gamma \nabla \chi^\infty \cdot \nabla \chi^\infty \:\d x
\\
&\leq& 
C \,\big( ||\nabla r||_{L^2(\Omega)^d} \;+\; ||\nabla \chi^\infty||_{L^2(\Omega)^d} \big)
||\nabla \chi^\infty||_{L^2(\Omega)^d} \\[0.5em]
&\leq& 
C \, ||\nabla \chi^\infty||_{L^2(\Omega)^d}^2,
\end{eqnarray*}
where the constant $C$ depends only on $\Omega$ and the bounds $\gamma^-$, $\gamma^+$ in \cref{cond_ell}.
The estimate \cref{est_cRN_cap} then follows from \cref{prop_1}. \par\medskip

\noindent\textit{Proof of~(ii):}
We again use the shorthand $\chi = \chi(\omega,k)$ for the solution to \cref{def_chi}.
From \cref{eq.Linftyvphie}, we have $0 \leq \chi(x) \leq 1$ a.e. in $\Omega$, so that:
\begin{equation*}
\capar^N(\omega,k) 
\: = \:
k \ds\int_{\omega} (1 - \chi) \:\d s \;\leq\; k \,|\omega| .
\end{equation*}
In order to obtain a lower bound for $\capar^N(\omega,k)$, we use the Cauchy-Schwarz inequality and the Poincar\'e's inequality \cref{eq.Poincare} to see that:
\begin{eqnarray*}
\ds\int_\omega \chi  \:\d s &\leq&
|\omega|^{1/2} \Big(\ds\int_\omega \chi^2 \:\d s \Big)^{1/2}
\;\leq\; C |\omega|^{1/2} ||\nabla \chi||_{L^2(\Omega)^d},
\end{eqnarray*}
where the constant $C$ is independent of $\omega$ and $k$. In addition, using $\chi$ as test function in \cref{def_chivarf} shows that
$$
\ds\int_\Omega \gamma |\nabla \chi|^2\:\d x
= k \ds\int_\omega \chi(1- \chi) \:\d s \;\leq\; \capar^N(\omega,k).
$$
Combining the previous two inequalities, we arrive at
$$
\ds\int_\omega \chi  \:\d s
\:\: \leq \:\: C |\omega|^{1/2} \big(\capar^N(\omega,k) \big)^{1/2}.
$$
We now deduce the following estimate
\begin{equation*}
\begin{array}{>{\displaystyle}cc>{\displaystyle}l}
\capar^N(\omega,k) &=& k \ds\int_\omega (1-\chi) \:\d s \\
&=& k \lvert \omega \lvert \: -\: k \int_\omega \chi \:\d s \\
&\geq& k|\omega| - C k |\omega|^{1/2} \big(\capar^N(\omega,k) \big)^{1/2}. 
\end{array}
\end{equation*}
We have thus proved that the following quadratic inequality holds true: 
\begin{multline*}
 X^2 +bX -c  \:\: \geq \:\: 0, \text{ with unknown } X = \big(\capar^N(\omega,k) \big)^{1/2} \geq 0, \\
 \text{ and coefficients } b = C k |\omega|^{1/2} > 0, \:\: c = k|\omega| > 0. 
 \end{multline*}
The latter can only hold if $X \geq \frac{-b + \sqrt{b^2 + 4c}}{2}$, which yields, after elementary calculations:
\begin{eqnarray*}
\capar^N(\omega,k) &\geq& 
4 k \, |\omega| 
\ds\frac{1}
{\left(\sqrt{4 + C^2k} + \sqrt{C^2 k} \right)^2},
\\
&\geq& C \, k |\omega|,
\end{eqnarray*}
where the constant $C$ depends on the assumed maximum bound $M$ for $k$.
This completes the proof.
\medskip

\noindent \textit{Proof of~(iii):}
Let $\omega$ be a fixed, strict subset of $\Gamma_N$ and let $k_\e$ be a sequence of positive numbers such that $k_\e \to \infty$.
We use the shorthands $\chi_\e = \chi(\omega,k_\e)$ and $\chi^\infty = \chi^\infty(\omega)$ for the respective solutions to~\cref{def_chi} and~\cref{def_chi_inf}.
At first, using $\chi_\e$ as test function in \cref{def_chi}, 
we deduce from \cref{est_cRN_cap} that
\begin{equation}\label{est_chi_cap}
\begin{array}{>{\displaystyle}cc>{\displaystyle}l}
\ds\int_\Omega \gamma |\nabla \chi_\e|^2 \:\d x &= &  k_\e \ds\int_{\omega} (1 - \chi_\e)\chi_\e \:\d s \\[1em]
& \leq & k_\e\ds\int_{\omega} (1- \chi_\e) \:\d s 
\\[1em]
&=& \capar^N(\omega,k_\e)  \\[1em]
&\leq&  C\capa(\omega),
\end{array}
\end{equation}
where the second line follows from \cref{eq.Linftyvphie} and the final line is a consequence of \cref{est_cRN_cap}.
The sequence $\chi_\e$ is therefore uniformly bounded in $H^1(\Omega)$. Hence, up to a subsequence, still denoted by $\e$, 
there exists a function $\chi^* \in H^1(\Omega)$
such that $\chi_\e \to \chi^*$ weakly in $H^1(\Omega)$.
Passing to the weak limit in the variational problem \cref{def_chivarf} for $\chi_\e$, it is easy to check that $\chi^*$ solves
\begin{equation}\label{def_chigen}
\left\{
\begin{array}{cccl}
-\dv(\gamma\nabla \chi^*) &=& 0 & \text{in } \Omega, \\[5pt]
\chi^* &=& 0 & \text{on } \Gamma_D, \\[5pt]
\gamma \frac{\partial \chi^*}{\partial n} &=& 0 & \text{on } \Gamma_N \setminus \overline{\omega},  \\[5pt]
\chi^* &=& 1 & \text{on } \omega,
\end{array}
\right. 
\end{equation}
that is, $\chi^* = \chi^\infty$. In addition, the uniqueness of this limit implies that
the whole sequence $\chi_\e$ weakly converges to $\chi^\infty$ in $H^1(\Omega)$.
Next, we compute:
\begin{equation*}
\begin{array}{>{\displaystyle}cc>{\displaystyle}l}
\capar^N(\omega,k_\e) &=& \int_{\partial \Omega} \gamma \frac{\partial \chi_\e}{\partial n}\:\d s  \\[1em]
&=&  \ds\int_{\partial \Omega}  \gamma \frac{\partial \chi_\e}{\partial n} \, \chi_\infty \:\d s \\[1em]
& = & \ds\int_\Omega \gamma \nabla \chi_\e \cdot \nabla \chi^\infty \:\d x.
\end{array}
\end{equation*}
Hence, 
$$\capar^N(\omega,k_\e)  \xrightarrow{\e \to 0} \int_\Omega \gamma |\nabla \chi^\infty|^2 \:\d x,$$
a quantity which is controlled by $\capa(\omega)$ up to constants which are independent of $\e$, see \cref{prop_1}. 
This terminates the proof of \cref{lim_cRN_cap}.

\subsection{Monotonicity of the Robin-Neumann capacity: proof of \cref{prop.monotonecaparN}}\label{sec.moncaprN}

\noindent \textit{Proof of~(i):} Let $\omega$ be a fixed subset of the Neumann region $\Gamma_N$, et let $0 < k_1 \leq k_2$. We denote by $\chi_1 := \chi(\omega,k_1)$ and $\chi_2 := \chi(\omega,k_2)$ the equilibrium distributions associated to these two situations, see \cref{def_chi}. We recall from the pointwise bounds \cref{eq.Linftyvphie} in \cref{prop.estvphie} that these functions satisfy 
$$0 \leq \chi_1 \leq 1 \text{ and } 0 \leq \chi_2 \leq 1 \text{ on } \overline\Omega.$$  
A simple calculation shows that the difference $r := \chi_2 - \chi_1$ between them is the unique $H^1(\Omega)$ solution to the following boundary value problem:
\begin{equation}\label{eq.monrNk} 
\left\{ \begin{array}{ccll}
- \dv(\gamma \nabla r)  &=& 0 & \textrm{in}\; \Omega,
\\[5pt]
r &=& 0 & \textrm{on}\; \Gamma_D,
\\[5pt]
\gamma\frac{\partial r}{\partial n} &=& 0 & \textrm{on}\; \Gamma_N \setminus \overline\omega,
\\[5pt]
\gamma\frac{\partial r}{\partial n} + k_2 r &=& (k_2-k_1) (1-\chi_1) & \textrm{on}\; \omega.
\end{array} \right.
\end{equation}

Since $k_1 \leq k_2$ and $(1-\chi_1)$ is non negative on $\overline\Omega$, the term $(k_2-k_1) (1-\chi_1)$ at the right-hand side of the Robin boundary condition in \cref{eq.monrNk} is non negative. Another application of the pointwise bounds \cref{eq.Linftyvphie} shows that:
\begin{equation}\label{eq.rnegmonrNk}
r = \chi_2 - \chi_1 \:\:\geq\:\: 0 \text{ on }\overline\Omega.
\end{equation}

Now recalling the definition \cref{eq.defcapar} of the Robin-Neumann capacity, we see that:
$$ \capar^N(\omega,k_2) - \capar^N(\omega,k_1 ) = \int_\omega \gamma \frac{\partial r}{\partial n} \:\d s.$$
Hence, bringing into play the Dirichlet equilibrium distribution $\chi^\infty \equiv \chi^\infty(\omega)$ defined by \cref{def_chi_inf}, we obtain:
$$
\begin{array}{>{\displaystyle}cc>{\displaystyle}l}
\capar^N(\omega,k_2) - \capar^N(\omega,k_1 ) &=& \int_\omega \gamma \frac{\partial r}{\partial n} \chi^\infty \:\d s \\[1em]
 &=& \int_{\partial \Omega} \gamma \frac{\partial r}{\partial n} \chi^\infty \:\d s \\[1em]
 &=& \int_{\partial \Omega} \gamma \frac{\partial \chi^\infty}{\partial n} r \:\d s, \\[1em]
 &=& \int_{\omega} \gamma \frac{\partial \chi^\infty}{\partial n} r \:\d s, \\[1em]
\end{array}
$$
where the third line follows after two successive applications of the Green's formula. 

Now, since $\chi^\infty \equiv 1$ on $\omega$ and $0 \leq \chi^\infty \leq 1$ on $\overline\Omega$, the normal derivative $\gamma \frac{\partial \chi^\infty}{\partial n}$ is non negative on $\omega$. 
Recalling \cref{eq.rnegmonrNk}, this terminates the proof of $(i)$.\par\medskip

\noindent \noindent \textit{Proof of~(ii):} Let $k$ be a positive real number and let $\omega_1 \subset \omega_2$ be two subsets of the Neumann region $\Gamma_N$. We now set $\chi_1 := \chi(\omega_1,k)$ and $\chi_2 := \chi(\omega_2,k)$ the equilibrium distributions, see again the definition \cref{def_chi}.

In this situation again, an elementary  calculation shows that the difference $r := \chi_2 - \chi_1$ between both functions is the unique $H^1(\Omega)$ solution to the following boundary value problem:
\begin{equation}\label{eq.monrNom} 
\left\{ \begin{array}{ccll}
- \dv(\gamma \nabla r)  &=& 0 & \textrm{in}\; \Omega,
\\[5pt]
r &=& 0 & \textrm{on}\; \Gamma_D,
\\[5pt]
\gamma\frac{\partial r}{\partial n} &=& 0 & \textrm{on}\; \Gamma_N \setminus \overline{\omega_2},
\\[5pt]
\gamma\frac{\partial r}{\partial n} + kr &=& k(1- \chi_1) & \textrm{on}\; \omega_2 \setminus \overline{\omega_1}, \\[5pt]
\gamma\frac{\partial r}{\partial n} + k r &=& 0 & \textrm{on}\; \omega_1.
\end{array} \right.
\end{equation}

Since the right-hand side of the Robin boundary condition satisfied by $r$ on $\omega_2$ is non negative, a simple adaptation of the proof of the pointwise bounds \cref{eq.Linftyvphie} reveals that:
\begin{equation}\label{eq.rnegmonrNom}
r = \chi_2 - \chi_1 \:\:\geq \:\: 0 \text{ on }\overline\Omega.
\end{equation}

Let us now use the fact that $\gamma \frac{\partial \chi_1}{\partial n} = 0$ on $\omega_2 \setminus \overline{\omega_1}$ to express the difference between the capacities $\capar^N(\omega_2,k)$ and $\capar^N(\omega_1,k) $ as:
$$ \capar^N(\omega_2,k) - \capar^N(\omega_1,k) =  \int_{\omega_2} \gamma \frac{\partial r}{\partial n}  \:\d s .$$
Like in the proof of point $(i)$, we now introduce the Dirichlet equilibrium distribution $\chi_2^\infty := \chi^\infty(\omega_2)$ defined by \cref{def_chi_inf} in this formula and apply Green's formula twice to obtain:
$$ 
\begin{array}{>{\displaystyle}cc>{\displaystyle}l}
\capar^N(\omega_2,k) - \capar^N(\omega_1,k)  &=&  \int_{\omega_2} \gamma \frac{\partial r}{\partial n} \chi^\infty_2 \:\d s \\[1em]
&=&  \int_{\omega_2} \gamma \frac{\partial \chi^\infty_2}{\partial n} r \:\d s \\[1em]
\end{array}
$$
Again, the claimed inequality in $(ii)$ follows from the combination of \cref{eq.rnegmonrNk} with the non negativity of the normal derivative $\gamma \frac{\partial \chi^\infty_2}{\partial n}$ on $\omega$ .\par\medskip


\section{Consistency of the expansions when $\Gamma_N$ is perturbed by a Robin and a Dirichlet boundary condition: Proof of \cref{thm5}} \label{sec_4}

\noindent In the previous sections, we have proved that, for a fixed subset $\omega_\e$ of the Neumann region $\Gamma_N$, the capacity $\capar^N(\omega,k_\e)$ behaves as $\capa(\omega)$ when the admittance parameter $k_\e$ tends to $\infty$. However, this behavior depends on  the considered set $\omega$. In this section, we show \cref{thm5}, which claims that, roughly speaking, it persists in the case of a sequence $\omega_\e$ of subsets that are vanishingly small, see \cref{thm5}. This indicates in particular that the asymptotic expansion \cref{asymp_pertN} associated to a Robin perturbation of the Neumann boundary is ``consistent'' with that \cref{asymp_uN} associated to a Dirichlet perturbation of this boundary, when the parameter $k_\e$ tends to $\infty$ ``sufficiently fast'' in the limit $\e \to 0$. 
This result rests on additional assumptions on the sets $\omega_\e$;
namely we assume the following:
\begin{equation}\label{eq.hypth5}
\begin{array}{cl}
\bullet \:\: \text{If } d =2, & \omega_\e \text{ is an arc of length } 2\e \text{ on } \partial \Omega, \\[1em]
\bullet \:\: \text{If } d =3, & \omega_\e \text{ is the image } \Phi(\D_\e)
\text{ of the planar disk } \D_\e = \{ x = (x_1, x_2, 0) \in \R^3, \quad x_1^2 + x_2^2 < \e^2\} \\[0.5em]
& \hspace{3.4cm} \text{ by a given smooth diffeomorphism }\Phi~: \R^3 \rightarrow \R^3. 
\end{array}
\end{equation}
To simplify the presentation, we limit ourselves to proving this result in the 2d case, in the particular configuration where $\omega_\e$ is the segment $(-\e/2,\e/2) \times \left\{0 \right\}$, pertaining to a completely flat part of the security compact $K_N \subset \Gamma_N$, while $\Omega$ coincides with the lower half-plane in the neighborhood of $0$. Moreover, we assume that the conductivity $\gamma$ identically equals $1$.

Throughout the proof, we use the shorthands $\chi_\e$ and $\chi_\e^\infty$ for the equilibrium distributions $\chi(\omega_\e,k_\e)$ and $\chi^\infty(\omega_\e)$ defined by \cref{def_chi_inf,def_chi}, respectively. 
We proceed in several steps. 
\par\medskip 

\noindent \textit{Step 1: We relate the capacities $\capar^N(\omega_\e,k_\e) $ and $\capa(\omega_\e)$, up to a remainder.}

To this end, let us recall from \cref{eq.defcapar} that
$$
\capar^N(\omega_\e,k_\e) \:\: = \:\: k_\e \ds\int_{\omega_\e} (1 - \chi_\e) \:\d s \:\: = \:\: \int_{\omega_\e} \frac{\partial \chi_\e}{\partial n}\:\d s,
$$
while the capacity $\capa(\omega_\e)$ is equivalent to the quantity
\begin{equation} \label{est_2}
 \ds\int_\Omega | \nabla \chi_\e^\infty|^2 \:\d x
\;=\; \ds\int_{\omega_\e} \frac{ \partial \chi_\e^\infty}{\partial n} \:\d s,
\end{equation}
up to a multiplicative constant that does not depend on $\e$, see \cref{prop_1}. 

We already know from~\cref{est_cRN_cap} that $\capar^N(\omega_\e,k_\e) \leq C \, \capa(\omega_\e)$, and we look for a reciprocal estimate when $k_\e$ is large enough. To achieve this, let us write:
\begin{equation*}\label{est_5a}
\begin{array}{>{\displaystyle}cc>{\displaystyle}l}
\capar^N(\omega_\e,k_\e)
& =&  \ds\int_{\omega_\e}  \frac{\partial \chi_\e}{\partial n}\chi_\e^\infty  \:\d s\\[1em] 
&=& 
\ds\int_\Omega \nabla \chi_\e \cdot \nabla \chi_\e^\infty \:\d x \\[1em] 
&=&
\ds\int_{\partial \Omega} \frac{\partial \chi_\e^\infty}{\partial n} \chi_\e \:\d s\\[1em] 
&=& 
\ds\int_{\omega_\e} \frac{\partial \chi_\e^\infty}{\partial n} \chi_\e \:\d s. 
\end{array}
\end{equation*} 
Let us now bring into play the quantity $\capa(\omega_\e)$ into this equality: 
\begin{equation}\label{est_5a}
\begin{array}{>{\displaystyle}cc>{\displaystyle}l}
\capar^N(\omega_\e,k_\e) &=&
\ds\int_{\omega_\e} \frac{\partial \chi_\e^\infty}{\partial n} \:\d s -  \ds\int_{\omega_\e} \frac{\partial \chi_\e^\infty}{\partial n} (1 - \chi_\e) \:\d s \\[1em] 
&=& 
\ds\int_\Omega |\nabla \chi_\e^\infty|^2 \:\d x  - \ds\int_{\omega_\e} \frac{\partial \chi_\e^\infty}{\partial n} (1 - \chi_\e) \: \d s \\[1em] 
&\geq&
C \capa(\omega_\e) + \ds\int_{\omega_\e} \frac{\partial \chi_\e^\infty}{\partial n} r_\e \:\d s,
\end{array}
\end{equation}
where $r_\e = \chi_\e - \chi_\e^\infty$ is the unique $H^1(\Omega)$ solution to the following boundary-value problem:
\begin{equation} \label{eq.chimchiinf}
\left\{
\begin{array}{cccl}
-\Delta r_\e &=& 0 & \text{in } \Omega, 
\\[5pt]
r_\e &=& 0 & \text{on } \Gamma_D,
\\[5pt]
\frac{\partial r_\e}{\partial n} &=& 0 & \text{on } \Gamma_N \setminus \overline{\omega_\e},
\\[5pt]
\frac{\partial r_\e}{\partial n} + k_\e r_\e &=& - \frac{\partial \chi_\e^\infty}{\partial n} & \text{on } \omega_\e.
\end{array}
\right.
\end{equation}
\par\medskip
To complete the proof, we thus need to estimate the term $R_\e :=\int_{\omega_\e} \frac{\partial \chi_\e^\infty}{\partial n} r_\e \:\d s$. We rewrite the latter by rescaling about  the vicinity of the origin $0 \in \omega_\e$:
\begin{equation}\label{eq.defRe}
R_\e = \int_{\omega_1} \frac{\partial \overline{\chi_\e^\infty}}{\partial n} \overline{r_\e} \:\d s,
\end{equation}
where $\omega_1$ is the segment $(-\frac12, \frac12) \times \left\{ 0 \right\}$ and we have set 
$$\overline{\chi_\e^\infty}(z) = \chi_\e^\infty (\e z) \text{ and }\overline{r_\e}(z) = r_\e(\e z), \text{ for } z \in \frac{1}{\e} \Omega.$$ 

\noindent \textit{Step 2: We estimate the rescaled function $\overline{\chi_\e^\infty}$.}

Throughout the rest of the proof, we denote by $B_r$ (resp. $B_r^-$) the ball (resp. the lower half-ball) with center $0$ and radius $r$ in $\R^d$.
Let us observe that a change of variables yields: 
$$ 
\begin{array}{>{\displaystyle}cc>{\displaystyle}l}
\lvert\lvert \nabla \overline{\chi_\e^\infty} \lvert\lvert_{L^2(B_2^-)^d}^2 &=&  \e^2 \int_{B_2^-} \lvert (\nabla \chi_\e^\infty) (\e z) \lvert^2 \:\d z \\[1em]
&=&  \int_{B_{2\e}^-} \lvert \nabla \chi_\e^\infty (z) \lvert^2 \:\d z \\[1em]
&\leq& \lvert\lvert \nabla \chi_\e^\infty \lvert\lvert_{L^2(\Omega)^d}^2  \\[1em]
&\leq & C \capa(\omega_\e).
\end{array}
$$
We now set $m_\e := \frac{1}{\lvert B_2^-\lvert}\int_{B_2^-} \chi_\e^\infty\:\d x$, whose value is less than $1$, owing to point (i) in \cref{prop.estvphie}.
By virtue of the Poincar\'e-Wirtinger inequality in $B_2^-$, it holds: 
$$ \begin{array}{>{\displaystyle}cc>{\displaystyle}l}
\lvert\lvert \overline{\chi_\e^\infty} - m_\e \lvert\lvert_{L^2(B_2^-)}& \leq & C \lvert\lvert \nabla \overline{\chi_\e^\infty} \lvert\lvert_{L^\infty(B_2^-)^d} \\[1em]
& \leq&  C \capa(\omega_\e)^{1/2}.
\end{array}$$
Moreover, the function $\overline{\chi_\e^\infty} -m_\e$ satisfies the following mixed boundary value problem in $B_2^-$: 
$$
\left\{
\begin{array}{cccl}
-\Delta \big( \overline{\chi_\e^\infty} - m_\e \big) &=& 0 & \text{in } B_2^-, 
\\[5pt]
\overline{\chi_\e^\infty} - m_\e   &=& 1-m_\e & \text{on } \omega_1,
\\[5pt]
\frac{\partial \big( \overline{\chi_\e^\infty} - m_\e \big) }{\partial n} &=& 0 & \text{on } \omega_2 \setminus \overline{\omega_1},
\\[5pt]
\big( \overline{\chi_\e^\infty} - m_\e \big) &=& g_\e& \text{on } \partial B_2 \cap B_2^-,
\end{array}
\right.
$$
where $g_\e$ lies in $H^{1/2}(\partial B_2 \cap B_2^-)$.
It follows from the theory of elliptic regularity that 
$\overline{\chi_\e^\infty} -m_\e \in H^s(B_{3/2}^-)$ for any $1 \leq s < 3/2$ and that
\begin{equation} \label{est_10}
\left\lvert\left\lvert \overline{\chi_\e^\infty } -m_\e \right\lvert\right\lvert_{H^s(B_{3/2}^-)} 
\:\: \leq\:\:  C\, \left\lvert\left\lvert\overline{\chi_\e^\infty} - m_\e\right\lvert\right\lvert_{H^1(B_{2}^-)} 
\;\leq\; C \capa(\omega_\e)^{1/2},
\end{equation}
with $C$ independent of~$\e$, see \cite{gilbarg2015elliptic,grisvard2011elliptic}. 
Invoking the trace Theorem, see for instance Theorem~1.5.2.1 in~\cite{grisvard2011elliptic},
we obtain that $\frac{\partial \chi_\e^\infty}{\partial n} \in H^{-s}(\partial B_{3/2}^-)$
for any~$0 < s \leq 1/2$, with the estimate:
\begin{eqnarray*} 
\left\lvert\left\lvert \frac{\partial \overline{\chi_\e^\infty }}{\partial n}
\right\lvert\right\lvert_{H^{-s} (\partial B_{3/2}^-)} 
&\leq& C \, \left\lvert\left\lvert \overline{\chi_\e^\infty } - m_\e 
\right\lvert\right\lvert_{H^1(B_2^-)}
\;\leq\; C \capa(\omega_\e)^{\frac{1}{2}}.
\end{eqnarray*}
As $\frac{\partial \overline{\chi_\e^\infty}}{\partial n}$
vanishes outside $\omega_1$, we conclude that 
$\frac{\partial \overline{\chi_\e^\infty}}{\partial n} \in \widetilde{H}^{-s}(\omega_1)$ with 
\begin{equation} \label{est_10b}
\left\lvert\left\lvert \frac{\partial \overline{\chi_\e^\infty }}{\partial n}\right\lvert\right\lvert_{\widetilde{H}^{-s} (\omega_1)} 
\:\: \leq \:\: C 
\left\lvert\left\lvert \frac{\partial \overline{\chi_\e^\infty }}{\partial n}
\right\lvert\right\lvert_{H^{-s} (\partial B_{3/2}^-)} 
\;\leq\; C \capa(\omega_\e)^{\frac{1}{2}},
\end{equation}
where $C$ is independent of $\e$.

Note that it is possible to directly control the $\widetilde{H}^{-s}(\omega_\e)$ norm of the normal derivative $\frac{\partial \chi_\e}{\partial n}$ of the unscaled function $\chi_\e$ from $\omega_\e$, up to a finer estimate. 
This result is not needed in the present proof, but since it is interesting in its own right, a statement and proof are included in \cref{sec.app}. 
\par\medskip

\noindent \textit{Step 3: We estimate the rescaled function $\overline{r_\e}$.}

We first note that 
\begin{equation} \label{est_grad_re}
\begin{array}{>{\displaystyle}cc>{\displaystyle}l}
||\nabla r_\e||_{L^2(\Omega)^d}   &\leq&||\nabla \chi_\e||_{L^2(\Omega)^d} + ||\nabla \chi_\e^\infty||_{L^2(\Omega)^d}\\[1em]
&\leq&\left( k_\e \ds\int_{\omega_\e} (1- \chi_\e)\chi_\e \:\d s \right)^{1/2} + C\capa(\omega_\e)^{1/2}\\[1em]
&\leq&C \left(\capar^N(\omega_\e,k_\e)^{1/2} + \capa(\omega_\e)^{1/2}\right) \\[1em]
& \leq & C\, \capa(\omega_\e)^{1/2},
\end{array}
\end{equation}
where we have used the pointwise bounds \cref{eq.Linftyvphie} to pass from the second line to the third one and \cref{est_cRN_cap} to arrive at the final line. 
Rescaling this inequality, we obtain 
\begin{eqnarray*}
\begin{array}{>{\displaystyle}cc>{\displaystyle}l}
||\nabla \overline{r_\e}||_{L^2(B_2^-)^d} &=& ||\nabla r_\e||_{L^2(B_{2\e}^-)^d} \\
 & \leq &  C\, \capa(\omega_\e)^{1/2}.
 \end{array}
\end{eqnarray*}
In addition, we have
\begin{eqnarray*}
||\overline{r_\e}||_{H^{1/2}(\omega_1)}^2 &\leq& ||\overline{r_\e}||_{H^{1}(B^-_1)}^2 \\
&=&  ||\overline{r_\e}||_{L^2(B^-_1)}^2 + \lvert\lvert \nabla \overline{r_\e} \lvert\lvert_{L^2(B^-_1)^2}^2 \\
&=&  \e^{-2} ||r_\e||_{L^2(B^-_\e)}^2 + \lvert\lvert \nabla r_\e \lvert\lvert_{L^2(B^-_\e)^2}^2 \\
&\leq& \e^{-2} ||r_\e||_{H^1( \Omega)}^2.
\end{eqnarray*}
Recalling that $|\omega_\e|=2\e$,
it follows from \cref{est_grad_re} and the Poincar\'e's inequality \cref{eq.Poincare} that
\begin{equation} \label{est_6}
\lvert\lvert \overline{r_\e} \lvert\lvert_{H^{1/2}(\omega_1)} \:\: \leq \:\: \ds\frac{C}{|\omega_\e|} \, \capa(\omega_\e)^{1/2}.
\end{equation}
\medskip

On the other hand, multiplying \cref{eq.chimchiinf} by $r_\e$ and integrating, we obtain
\begin{equation*}
\ds\int_\Omega |\nabla r_\e|^2  \:\d x + k_\e \int_{\omega_\e} r_\e^2 \:\d s = -\ds\int_{\partial \Omega} \frac{\partial \chi_\e^\infty}{\partial n} r_\e \:\d s,
\end{equation*}
and so:
\begin{equation}\label{est_7}
\begin{array}{>{\displaystyle}cc>{\displaystyle}l}
||r_\e||_{L^2(\omega_\e)} &\leq&  k_\e^{-1/2}  \left\lvert \left\lvert \frac{\partial \chi_\e^\infty}{\partial n} 
\right\lvert\right\lvert_{H^{-1/2}(\partial \Omega)} ^{1/2} 
\, ||r_\e||_{H^{1/2}(\partial \Omega)} ^{1/2} 
\\[1em]
&\leq& C \, k_\e^{-1/2} \,  ||\chi_\e^\infty||_{H^{1}(\Omega)} ^{1/2} 
\; ||\nabla r_\e||_{L^2(\Omega)}^{1/2} 
\\[1em] 
&\leq& C \, k_\e^{-1/2} \capa(\omega_\e)^{1/2},
\end{array}
\end{equation}
where we have used \cref{est_grad_re}
to obtain the final inequality.
A simple change of variables now yields: 
$$\lvert\lvert \overline{r_\e} \lvert\lvert_{L^2(\omega_1)}  \:\: \leq  \:\: \ds\frac{C}{\lvert\omega_\e\lvert^{1/2} k_\e^{1/2}} \capa(\omega_\e)^{1/2}. $$

By interpolation, we deduce from \cref{est_6} and \cref{est_7}
that $\overline{r_\e} \in H^{1/4}(\omega_1)$ with:
\begin{equation} \label{est_8}
||\overline{r_\e}||_{H^{1/4}(\omega_1)} 
\:\: \leq \:\: 
\ds\frac{C}{\lvert\omega_\e\lvert^{3/4}k_\e^{1/4}} \, \capa(\omega_\e)^{1/2}.
\end{equation}
\medskip

\noindent \textit{Step 4: Conclusion.}

It follows from the definition \cref{eq.defRe} of $R_\e$ and from the estimates~\eqref{est_10b}and~\eqref{est_8} that: 
 $$ 
 \begin{array}{>{\displaystyle}cc>{\displaystyle}l}
 R_\e &\leq&  \left\lvert\left\lvert \frac{\partial \overline{\chi_\e^\infty }}{\partial n}\right\lvert\right\lvert_{\widetilde{H}^{-1/4} (\omega_1)}  \: ||\overline{r_\e}||_{H^{1/4}(\omega_1)}   \\[1em]
 &\leq& \frac{C}{\lvert\omega_\e\lvert^{3/4}k_\e^{1/4}} \, \capa(\omega_\e).
\end{array}
 $$
Hence, if $k_\e$ tends to infinity so fast that $k_\e \lvert \omega_\e\lvert^3 \to \infty$, 
the remainder $R_\e$ tends to $0$.
This terminates the proof.
%
%
%
%
%
\section{Proofs of the results about Robin perturbation of the Dirichlet boundary $\Gamma_D$}  \label{sec_5}


\noindent In this section, we consider the situation where the perturbation of the background potential $u_0$ in \cref{eq.conducbg2} is induced by a replacement of the Dirichlet boundary condition by a Robin boundary condition on a subset $\omega_\e$ of the region $\Gamma_D \subset \partial \Omega$: the perturbed potential $u_\e^{D,k_\e} \in H^1(\Omega)$, denoted by $u_\e$ throughout this section, solves the boundary-valued problem \cref{eq_pertD},
so that the difference $r_\e = u_\e - u_0$ is the $H^1(\Omega)$ solution to the following problem:
\begin{equation} \label{eq_reD}
\left\{ \begin{array}{ccll}
- \textrm{div}( \gamma \nabla r_\e ) &=& 0 & \textrm{in}\; \Omega,
\\ 
r_\e &=& 0 & \textrm{on}\; \Gamma_D \setminus \omega_\e,
\\ 
\gamma \frac{\partial r_\e}{\partial n} &=& 0 & \textrm{on}\; \Gamma_N,
\\ 
\gamma \frac{\partial r_\e}{\partial n} + k_\e r_\e &=& - \gamma \frac{\partial u_0}{\partial n}  
& \textrm{on}\; \omega_\e.
\end{array} \right.
\end{equation}\par
\medskip
In this setting, we aim to derive an asymptotic expansion of $r_\e$ as $\e\to 0$. 


\subsection{Preliminary estimates}

\noindent Let $\omega_\e$ be a sequence of subsets of the Dirichlet region $\Gamma_D$ that are strictly contained in the latter, in the sense that they all belong to a fixed compact $K_D \subset \Gamma_D$, i.e. \cref{eq.assumsepKD} holds true; let also $k_\e > 0$ be a sequence of positive numbers. 
In order to estimate the difference $r_\e$ in \cref{eq_reD}, we first consider, more generally, the solution $v_\e$ to the following boundary-value problem:
\begin{equation} \label{def_v}
\left\{ \begin{array}{ccll}
- \textrm{div}( \gamma \nabla v_\e ) &=& 0 & \textrm{in}\; \Omega,
\\ 
v_\e &=& 0 & \textrm{on}\; \Gamma_D \setminus \omega_\e,
\\ 
\gamma \frac{\partial v_\e}{\partial n}  &=& 0 & \textrm{on}\; \Gamma_N,
\\ 
\gamma \frac{\partial v_\e}{\partial n}  + k_\e v_\e &=& g & \textrm{on}\; \omega_\e,
\end{array} \right.
\end{equation}
where $g$ is a given function in ${\mathcal C}^\infty(\R^d)$. Equivalently, this function satisfies the following variational problem:
\begin{equation}\label{eq.varfcapfuncrobgenD}
\text{Find } v_\e \in H^1_{\Gamma_D}(\Omega) 
\text{ s.t. } \forall v \in H^1_{\Gamma_D}(\Omega), \quad 
\int_\Omega \gamma\nabla v \cdot \nabla v \:\d x 
+k_\e\int_{\omega_\e} v_\e v \:\d s =  \int_{\omega_\e}  gv \:\d s.
\end{equation}

Before starting, let us recall the following version of the Poincar\'e's inequality for functions vanishing on $\Gamma_D \setminus K_D$: there exists a constant $C >0$ depending on $\Omega$, $\Gamma_N$, $\Gamma_D$, $K_D$ and the bounds $\gamma^-$, $\gamma^+$ on the conductivity $\gamma$ in \cref{cond_ell} such that:
\begin{equation}\label{eq.PoincareD}
\forall v \in H^1_{\Gamma_D \setminus K_D}(\Omega), \quad  \frac{1}{C} \lvert\lvert v \lvert\lvert_{H^1(\Omega)}^2 \:\: \leq \:\:  \int_\Omega \gamma\lvert\nabla v \lvert^2 \:\d x  \:\: \leq \:\:  \lvert\lvert v \lvert\lvert_{H^1(\Omega)}^2.
\end{equation}

We now state a lemma containing preliminary estimates about $v_\e$.

\begin{lemma} \label{lemma_est_re}
Let $\omega_\e$ be a sequence of subsets of $\Gamma_D$ satisfying 
the separation assumption~\cref{eq.assumsepKD} and let $k_\e$ be any sequence 
of positive real numbers; let~$g\in \calC^\infty(\R^d)$, and let~$v_\e$ 
be the unique solution to~\cref{def_v}. 
Then,
\begin{enumerate}[(i)] 
\item The following pointwise bounds hold true:
\begin{equation}\label{eq.LinftyveD}
\frac{1}{k_\e} \min \left(0, \inf\limits_{x\in \overline{\Omega}} g(x) \right)  \: \leq \:  v_\e(x) \: \leq \: \frac{1}{k_\e} \max\left(0,\sup\limits_{x\in \overline{\Omega}} g(x)\right) \text{ for a.e. } x\in \overline\Omega.
 \end{equation} 


\item
The following $H^1$ estimate holds true:
\begin{equation} \label{est_H1ve}
||v_\e||_{H^1(\Omega)} \;\leq\; C ||g||_\infty \, \capar^D(\omega_\e,k_\e)^{1/2},
\end{equation}
where the constant $C$ depends on $\Omega$, $\Gamma_N$, $\Gamma_D$ and the compact $K_D$, but not on $\omega_\e$ or $k_\e$. 
\item The following improved $L^2$ estimate holds true:
\begin{equation} \label{est_L2ve}
||v_\e||_{L^2(\Omega)} \;\leq\; C ||g||_\infty \, \capar^D(\omega_\e,k_\e)^{3/4},
\end{equation}
where, again, $C$ depends only $\Omega$, $\Gamma_N$, $\Gamma_D$ and the compact $K_D$.
\end{enumerate}
\end{lemma}

\begin{proof}
\textit{Proof of (i).} The proof is similar to that of $(i)$ in \cref{prop.estvphie}. Let $M$ be the quantity at the right-hand side of \cref{eq.LinftyveD}; setting
$\psi_\e(x) = \max(v_\e(x)-M,0)$ as test function in the variational problem \cref{eq.varfcapfuncrobgenD} satisfied by $v_\e$, we see that:
$$
\int_\Omega \gamma \lvert \nabla v_\e \lvert^2 \mathds{1}_{v_\e \geq M} \:\d x  + k_\e \int_{\omega_\e} v_\e (v_\e-M) \mathds{1}_{v_\e \geq M} \:\d s = \int_{\omega_\e} g (v_\e-M) \mathds{1}_{v_\e \geq M} \:\d s,
$$
and so:
$$
\int_\Omega \gamma \lvert \nabla v_\e \lvert^2 \mathds{1}_{v_\e \geq M} \:\d x  = \int_{\omega_\e} (g - k_\e v_\e) (v_\e-M) \mathds{1}_{v_\e \geq M} \:\d s \:\leq \: 0,
$$
as follows from the very definition of $M$. Hence, $v_\e(x) \leq M$ for a.e. $x \in \overline\Omega$, which is exactly the right-hand inequality in \cref{eq.LinftyveD}.
The left-hand inequality in this statement follows by a similar argument. 
\medskip

\noindent \textit{Proof of (ii).} We actually prove the stronger estimate 
\begin{equation}\label{eq.stronger_est_H1ve}
 \int_\Omega \gamma \lvert \nabla v_\e \lvert^2 \:\d x + k_\e \int_{\omega_\e} v_\e^2 \:\d s \;\leq\; C ||g||_\infty \, \capar^D(\omega_\e,k_\e),
 \end{equation}
which directly implies \cref{est_H1ve} upon application of the Poincar\'e's inequality \cref{eq.PoincareD} for functions vanishing on $\Gamma_D \setminus K_D$. 

To achieve this, we first take $v = v_\e$ as test function in the variational problem \cref{eq.varfcapfuncrobgenD} for $v_\e$, which yields: 
\begin{equation}\label{eq.aprioriH1veDbef}
\begin{array}{>{\displaystyle}cc>{\displaystyle}l}
 \int_\Omega \gamma \lvert \nabla v_\e \lvert^2 \:\d x + k_\e \int_{\omega_\e} v_\e^2 \:\d s &=& \int_{\omega_\e}  g v_\e \:\d s \\[1em]
  &\leq& \lvert\lvert g \lvert\lvert_{\calC(\overline\Omega)} \int_{\omega_\e} \lvert v_\e \lvert \:\d s.
 \end{array}
 \end{equation}
 
 Next, bringing into play the Robin-Dirichlet equilibrium distribution $\xi_\e := \xi(\omega_\e,k_\e)$ defined by \cref{def_xi}, a simple calculation yields:
 $$
\begin{array}{>{\displaystyle}cc>{\displaystyle}l}
\int_{\omega_\e} \lvert v_\e \lvert \:\d s &=& \int_{\omega_\e} \gamma \frac{\partial \xi_\e}{\partial n} \lvert v_\e\lvert \:\d s + \int_{\omega_\e} k_\e \xi_\e \lvert v_\e\lvert \:\d s \\[1em]
&=& \int_{\partial \Omega} \gamma \frac{\partial \xi_\e}{\partial n} \lvert v_\e\lvert \:\d s + \int_{\omega_\e} k_\e \xi_\e \lvert v_\e\lvert \:\d s \\[1em]
&=& \int_{\Omega} \gamma \nabla \xi_\e \cdot \nabla \lvert v_\e\lvert \:\d x + \int_{\omega_\e} k_\e \xi_\e \lvert v_\e\lvert \:\d s.
 \end{array}
 $$
The Cauchy-Schwarz inequality then implies that: 
 $$ \int_{\omega_\e} \lvert v_\e \lvert \:\d s \:\: \leq \:\: \left( \int_\Omega \gamma \lvert \nabla \xi_\e \lvert^2 \:\d x + k_\e \int_{\omega_\e} \xi_\e^2 \:\d s\right)^{\frac12} \left( \int_\Omega \gamma \lvert \nabla v_\e \lvert^2 \:\d x + k_\e \int_{\omega_\e} v_\e^2 \:\d s\right)^{\frac12}.$$
Inserting this inequality into \cref{eq.aprioriH1veDbef} and rearranging, we arrive at: 
$$  \int_\Omega \gamma \lvert \nabla v_\e \lvert^2 \:\d x + k_\e \int_{\omega_\e} v_\e^2 \:\d s \:\: \leq \:\: \lvert\lvert g \lvert\lvert_{\calC(\overline\Omega)}^2 \left( \int_\Omega \gamma \lvert \nabla \xi_\e \lvert^2 \:\d x + k_\e \int_{\omega_\e} \xi_\e^2 \:\d s\right).$$
Eventually, recalling the definition \cref{def_cRD} of the Robin-Dirichlet capacity $\capar^D(\omega_\e,k_\e)$, this proves the desired result \cref{eq.stronger_est_H1ve}, and thus \cref{est_H1ve}.

For further reference, note that we have proved in passing the following estimate:
\begin{equation}\label{eq.furtheruseestH1}
 \int_{\omega_\e} \lvert v_\e \lvert \:\d s \:\: \leq \:\: C \lvert\lvert g \lvert\lvert_{\calC(\overline\Omega)} \capar^D(\omega_\e,k_\e).
\end{equation}
\par\medskip

\noindent \textit{Proof of (iii).} We adapt the duality argument in the proof of point $(iii)$ in \cref{prop.estvphie} to the present situation.
Let $w_\e$ be the $H^1(\Omega)$ solution to the following boundary-value problem:
\begin{equation}\label{eq.pbpeD}
\left\{
\begin{array}{cl}
-\dv(\gamma\nabla w_\e) = -v_\e & \text{in } \Omega, \\
w_\e = 0 & \text{on } \Gamma_D,\\
\gamma \frac{\partial w_\e}{\partial n} = 0 & \text{on } \Gamma_N.
\end{array}
\right.
\end{equation} 
We introduce two open sets $\calO_1 \Subset \calO_2$ such that $K_D \subset \calO_1$ and $\overline{\calO_2} \cap \overline{\Gamma_N} = \emptyset$, 
as well as a smooth cut-off function $\eta$ satisfying $\eta \equiv 1$ on $\calO_1$ and $\eta\equiv 0$ on $\R^d \setminus \calO_2$.  
Here again, classical elliptic regularity theory and the Sobolev embedding theorem imply 
that $w_\e \in H^3(\calO_2 \cap \overline\Omega)$ with
\begin{equation}\label{eq.estC1we}
\lvert\lvert w_\e \lvert\lvert_{\calC^1(\calO_2 \cap \overline\Omega)} \: \leq \: C \lvert\lvert w_\e \lvert\lvert_{H^3(\calO_2 \cap \Omega)} \: \leq \: C \lvert\lvert v_\e \lvert\lvert_{H^1(\Omega)} \:\leq \: C \, ||g||_\infty \capar^D(\omega_\e,k_\e)^{1/2}.
\end{equation}
Now, repeated integrations by parts yield:
\begin{equation*} \label{est_4D}
\begin{array}{>{\displaystyle}cc>{\displaystyle}l}
\int_\Omega v_\e^2 \:\d x&=&  \int_\Omega \textrm{div}(\gamma \nabla w_\e) v_\e \:\d x \\[1em] 
&=& \int_{\partial\Omega} \gamma \frac{\partial w_\e}{\partial n} v_\e \:\d s  - \int_{\Omega} \gamma \nabla w_\e \cdot \nabla v_\e \:\d x\\[1em]
&=&  \int_{\partial\Omega} \gamma \frac{\partial w_\e}{\partial n} v_\e \:\d s  - \int_{\partial \Omega} \ \gamma \frac{\partial v_\e}{\partial n} w_\e \:\d s,\\[1em]
\end{array}
\end{equation*}
and so, using the boundary conditions satisfied by $v_\e$ and $w_\e$, see \cref{def_v,eq.pbpeD}:
$$ \int_\Omega v_\e^2 \:\d x =  \int_{\omega_\e} \gamma \frac{\partial w_\e}{\partial n} v_\e \:\d s.$$
Next, we estimate: 
$$
\begin{array}{>{\displaystyle}cc>{\displaystyle}l}
 \int_\Omega v_\e^2 \:\d x & \leq & \left\lvert \left\lvert w_\e \right\lvert \right\lvert_{\calC^1(\calO_2 \cap \overline\Omega)} \int_{\omega_\e} \lvert v_\e\lvert\:\d s\\[1em]
 &\leq& C \lvert\lvert g \lvert\lvert_{\calC^1(\overline\Omega)} \capar^D(\omega_\e,k_\e)^{\frac32},
\end{array}
$$
where we have used the $\calC^1$ estimate \cref{eq.estC1we} about $w_\e$ and the intermediate result \cref{eq.furtheruseestH1} from the proof of $(ii)$ in passing from the first line to the second one.
This terminates the proof of the desired estimate \cref{est_L2ve}.

\end{proof}

\subsection{Proof of \cref{thm2}: asymptotic expansion of $r_\e$}

\noindent The goal of this section is to prove the general asymptotic formula \cref{asymp_uD} for the perturbed voltage potential $u_\e$ by the replacement of the homogeneous Dirichlet boundary condition by a Robin boundary condition with arbitrary parameter $k_\e$ on a subset $\omega_\e \subset \Gamma_D$ of arbitrary shape. 

\begin{proof}[Proof of \cref{thm2}]

We follow again the approach of \cite{bonnetier_dapogny_vogelius}, and we first establish a representation formula for the error $r_\e = u_\e-u_0$ between the perturbed and background potentials, solution to \cref{eq_reD} at a fixed point $x \in \Omega$. It follows from the definition \cref{eq_Green} of the Green's function $G(x,y)$ and two successive integration by parts that:
\begin{eqnarray*}
r_\e(x) &=& 
-\ds\int_\Omega \textrm{div}_y(\gamma(y) \nabla_y G(x,y)) r_\e(y) \:\d y
\\
&=& \ds\int_\Omega \gamma(y) \nabla_y G(x,y) \cdot \nabla r_\e(y)  \:\d y
\;-\; \ds\int_{\partial \Omega} \gamma(y) \ds\frac{\partial G}{\partial n_y} (x,y) r_\e(y) \:\d s(y)
\\
&=&
-\ds\int_{\partial \Omega} \gamma(y) \ds\frac{\partial G}{\partial n_y}(x,y) r_\e(y)  \:\d s(y)
\;+\;
\ds\int_{\partial \Omega} \gamma(y) \frac{\partial r_\e}{\partial n}(y) G(x,y) \:\d s(y).
\end{eqnarray*}
In view of the boundary conditions satisfied by $G(x,\cdot)$ and $r_\e$, the above equation reduces to
\begin{equation} \label{eq_reprD}
r_\e(x) =- \int_{\omega_\e} \gamma(y) \ds\frac{\partial G}{\partial n_y} (x,y)r_\e(y)\:\d s(y).
\end{equation}
Observing that the function $y \mapsto \gamma(y) \frac{\partial G}{\partial n_y}(x,y)$ is smooth on $\omega_\e$, we introduce a smooth function $\phi \in \calC_c^\infty(\R^d)$ with compact support, 
which coincides with $\gamma(\cdot)G(x,\cdot)$ on $K_D$ and vanishes on a neighborhood of $\Gamma_N$ that does not intersection $K_D$. 


We now insert the Robin-Dirichlet equilibrium distribution $\xi_\e = \xi(\omega_\e,k_\e)$ defined in \cref{def_xi} into \cref{eq_reprD} via its Robin boundary condition on $\omega_\e$. An elementary calculation based on repeated applications of the Green's formula then yields: 
$$ 
\begin{array}{>{\displaystyle}cc>{\displaystyle}l}
r_\e(x) &=& - \int_{\omega_\e} \phi(y) r_\e(y)\:\d s(y) \\[1em]
&=& -\int_{\omega_\e} \gamma \frac{\partial \xi_\e}{\partial n} \phi r_\e \:\d s - k_\e \int_{\omega_\e}  \xi_\e \phi r_\e \:\d s\\[1em]
&=& -\int_{\partial \Omega} \gamma \frac{\partial \xi_\e}{\partial n} \phi r_\e \:\d s - k_\e \int_{\omega_\e}  \xi_\e \phi r_\e \:\d s\\[1em]
&=& -\int_{\Omega} \gamma \nabla \xi_\e \cdot \nabla (\phi r_\e) \:\d x - k_\e \int_{\omega_\e}  \xi_\e \phi r_\e \:\d s\\[1em]
&=& -\int_{\Omega} \gamma \nabla (\phi \xi_\e) \cdot \nabla r_\e \:\d x - k_\e \int_{\omega_\e}  \xi_\e \phi r_\e \:\d s + R_\e\\[1em]
&=& -\int_{\omega_\e} \gamma  \frac{\partial r_\e}{\partial n} \phi \xi_\e \:\d s - k_\e \int_{\omega_\e}  \xi_\e \phi r_\e \:\d s + R_\e,
\end{array}
$$
and so:
\begin{equation}\label{eq.estreDfinal}
r_\e(x) = \int_{\omega_\e} \gamma  \frac{\partial u_0}{\partial n} \phi \xi_\e \:\d s + R_\e.
\end{equation}

Here we have used the fact that $r_\e$ and $\xi_\e$ are solutions to \cref{eq_reD}
and \cref{def_xi} respectively, and we have set:
$$ R_\e = -\int_\Omega (\gamma \nabla \xi_\e \cdot \nabla \phi) \: r_\e \:\d x +  \int_\Omega (\gamma \nabla r_\e \cdot \nabla \phi) \:\xi_\e \:\d x.$$
Now applying the ``classical'' $H^1$ and improved $L^2$ estimates about $r_\e$ and $\xi_\e$ contained in \cref{lemma_est_re} to control the remainder $R_\e$, we obtain:
\begin{equation} \label{eq.estReD}
\begin{array}{>{\displaystyle}cc>{\displaystyle}l}
|R_\e| &\leq& \lvert\lvert \phi \lvert\lvert_{\calC^1(\R^d)} \Big( \lvert\lvert \nabla \xi_\e \lvert\lvert_{L^2(\Omega)^d} \lvert\lvert r_\e \lvert\lvert_{L^2(\Omega)}  +  \lvert\lvert \nabla r_\e \lvert\lvert_{L^2(\Omega)^d} \lvert\lvert \xi_\e \lvert\lvert_{L^2(\Omega)}\Big) \\[1em]
&\leq&  C \,\lvert\lvert \phi \lvert\lvert_{\calC^1(\R^d)}  \: \capar^D(\omega_\e,k_\e)^{\frac54}.
\end{array}
\end{equation}
\medskip

On a different note, as another consequence of \cref{lemma_est_re}, the following estimate holds true:
for any $\psi \in {\mathcal C}^1(\Gamma_D)$
$$
\forall \psi \in \calC(\partial \Omega), \quad \frac{1}{\capar^D(\omega_\e,k_\e)}
\int_{\omega_\e}  \xi_\e  \psi \:\d s
\:\: \leq\:\: C ||\psi||_{\calC(\partial \Omega)}.
$$
Invoking the Banach-Alaoglu Theorem, there exists a subsequence of the $\e$ (which is not relabelled for simplicity) 
and a non negative Radon measure $\mu$ supported in $K_D$ such that for any 
$\phi \in {\mathcal C}^1_0(\Gamma_D)$,
$$\text{For any function } \phi \in \calC^1(\overline{\Gamma_D}), \quad \frac{1}{\capar^D(\omega_\e,k_\e)}
\ds\int_{\omega_\e} 
\xi_\e \phi 
\:\:\xrightarrow{\e\to 0} \:\:  \left\langle \mu, \phi \right\rangle.
$$
Combining \cref{eq.estreDfinal} with \cref{eq.estReD}, we therefore obtain:
$$r_\e(x) =  \capar^D(\omega_\e,k_\e) \, 
\left\langle\mu, \gamma \frac{\partial u_0}{\partial n} \gamma\frac{\partial G}{\partial n_y}(x,\cdot)\right\rangle
\;+\; \o\left(\capar^D(\omega_\e,k_\e)\right), 
$$
which terminates the proof.

\end{proof}

\subsection{Proof of \cref{thm4}}

\noindent We now turn to analyze the behavior of the Robin-Dirichlet capacity $\capar^D(\omega,k)$ depending on the value of the admittance parameter $k$. 
In particular,  in the regime where $k \to 0$, we compare this quantity with that $e(\omega)$ measuring the smallness of the support $\omega$ of a replacement of the homogeneous Dirichlet condition $u=0$ with a Neumann condition, in line with the intuition that the Robin condition $\gamma \frac{\partial u}{\partial n} + ku = 0$ formally reduces to a change with the Neumann condition $\gamma\frac{\partial u}{\partial n} = 0$ in this regime.
\medskip

\noindent\textit{Proof of (i):}
Recalling the definitions \cref{eq.capfuncD} and \cref{def_xi} of the equilibrium distributions $\xi^0 =\xi^0(\omega)$
and $\xi = \xi(\omega,k)$ and that \cref{def_cRD} of the Robin-Dirichlet capacity, we see that:
$$
\begin{array}{>{\displaystyle}cc>{\displaystyle}l}
\capar^D(\omega,k)  & = & \int_\omega \xi \: \d s \\[1em] 
&=&  \int_\omega \gamma \frac{\partial \xi^0}{\partial n} \xi \:\d s \\[1em]
&=&  \int_\Omega  \gamma \nabla \xi^0 \cdot \nabla \xi \:\d x \\[1em]
&\leq& \lvert\lvert\nabla \xi^0 \lvert\lvert_{L^2(\Omega)^d} \lvert\lvert\nabla \xi \lvert\lvert_{L^2(\Omega)^d}
\end{array}
$$
and the desired estimate \cref{est_cRD_e} follows from \cref{prop_1} and \cref{lemma_est_re}.
\medskip

\noindent\textit{Proof of (ii):}
Let $\omega \Subset \Gamma_D$ and $k > 0$ be given. A simple application of the Cauchy-Schwarz inequality reveals that:
$$
\capar^D(\omega,k)  \:=\: \int_\omega \xi \:\d s \;\leq\; |\omega|^{1/2} \, ||\xi||_{L^2(\omega)}.
$$
On a different note, the very definition \cref{def_cRD} of the Robin-Dirichlet capacity shows that the estimate
$$
k ||\xi||_{L^2(\omega)}^2 \:\: \leq \:\:\capar^D(\omega,k).
$$
Combining both inequalities and rearranging, we arrive at: 
$$\capar^D(\omega,k) \leq \frac{1}{k}|\omega|,$$
as desired.
\medskip

\noindent\textit{Proof of (iii):}
Let $\omega_\e$ be a sequence of subsets of the fixed compact $K_D \subset \Gamma_D$, and let $k_\e$ be a sequence of positive real numbers such that $k_\e \to 0$ as $\e \to 0$. 
Let $\xi_\e^0 = \xi^0(\omega_\e,k_\e)$ and $\xi_\e = \xi(\omega_\e,k_\e)$ be the equilibrium distributions characterized by \cref{eq.capfuncD} and \cref{def_xi}, respectively. 
Then, we calculate:
$$
\begin{array}{>{\displaystyle}cc>{\displaystyle}l}
e(\omega_\e) &\leq &C \int_\Omega \gamma |\nabla \xi_\e^0|^2 \:\d s \\[1em] 
&=&C \int_{\omega_\e} \xi_\e^0 \:\d s \\[1em] 
&=& C \int_{\omega_\e} \Big(\gamma\frac{\partial\xi_\e}{\partial n} + k_\e \xi_\e \Big)  \xi_\e^0 \:\d s \\[1em]
&=& C \left( \int_{\partial \Omega}\gamma\frac{\partial\xi_\e}{\partial n}  \xi_\e^0 \:\d s +  \int_{\omega_\e} k_\e \xi_\e \xi_\e^0 \:\d s \right).
\end{array}
$$ 
Now using the Green's formula and the Cauchy-Schwarz inequality, we arrive at:
$$
\begin{array}{>{\displaystyle}cc>{\displaystyle}l}
e(\omega_\e) &\leq &\Big( \int_\Omega \gamma |\nabla \xi_\e^0|^2 \:\d x  + k_\e \int_{\omega_\e}  \lvert \xi_\e^0 \lvert^2 \:\d s \Big)^{1/2}\Big( \int_\Omega \gamma |\nabla \xi_\e|^2 \:\d x  + k_\e \int_{\omega_\e}  \lvert \xi_\e \lvert^2 \:\d s \Big)^{1/2}\\[1em]
&\leq& C\, (1 + k_\e)^{1/2}\, ||\xi_\e^0||_{H^1(\Omega)} \capar^D(\omega_\e,k_\e)^{1/2},
\end{array}
$$
where we have used the continuity of the trace on $H^1_{\Gamma_D}$ to pass form the first line to the second one.  
Invoking again \cref{prop_1} together with \cref{est_cRD_e}, we obtain the desired estimate.

\subsection{Monotonicity of the Robin-Dirichlet capacity: proof of \cref{prop.monotonecaparD}}\label{sec.moncaprD}

\noindent \textit{Proof of~(i):} Let $\xi_1 := \xi(\omega,k_1)$ and $\xi_2 := \xi(\omega,k_2)$ be the versions of the equilibrium distribution in \cref{def_xi} associated to the Robin conditions parametrized by $k_1$ and $k_2$, respectively. 
The difference $r := \xi_2 - \xi_1$ between both functions is the unique $H^1(\Omega)$ solution to the following boundary value problem:
\begin{equation}\label{eq.monrDk} 
\left\{ \begin{array}{ccll}
- \dv(\gamma \nabla r)  &=& 0 & \textrm{in}\; \Omega,
\\[5pt]
r &=& 0 & \textrm{on}\; \Gamma_D,
\\[5pt]
\gamma\frac{\partial r}{\partial n} &=& 0 & \textrm{on}\; \Gamma_N \setminus \overline\omega,
\\[5pt]
\gamma\frac{\partial r}{\partial n} + k_2 r &=& (k_1-k_2) \xi_1 & \textrm{on}\; \omega.
\end{array} \right.
\end{equation}

The ordering $k_1 \leq k_2$ and the non negativity of $\xi_1$ imply that the term $(k_1-k_2) \chi_1$ at the right-hand side of the Robin boundary condition in \cref{eq.monrDk} is non positive. Invoking the pointwise bounds \cref{eq.LinftyveD}, we obtain:
$$
r = \xi_2 - \xi_1  \:\: \leq \:\: 0 \text{ on }\overline\Omega.
$$

We now readily infer from the definition \cref{def_cRD} of the Robin-Dirichlet capacity that:
$$ 
\capar^D(\omega,k_2) - \capar^D(\omega,k_1 ) \:\: = \:\: \int_{\omega} (\xi_2-\xi_1) \:\d s \:\:\leq \:\: 0,
$$
which is the expected conclusion.\par\medskip

\noindent \noindent \textit{Proof of~(ii):} Let now $k$ be a positive real number and let $\omega_1 \subset \omega_2$ be two subsets of the Dirichlet region $\Gamma_D$. We set $\xi_1 := \xi(\omega_1,k)$ and $\xi_2 := \xi(\omega_2,k)$.
An elementary  calculation shows that the difference $r := \xi_2 - \xi_1$ between both functions is the unique $H^1(\Omega)$ solution to the following boundary value problem:
\begin{equation}\label{eq.monrDom} 
\left\{ \begin{array}{ccll}
- \dv(\gamma \nabla r)  &=& 0 & \textrm{in}\; \Omega,\\[5pt]
\gamma\frac{\partial r}{\partial n} &=& 0 & \textrm{on}\; \Gamma_N,\\[5pt]
r &=& 0 & \textrm{on}\; \Gamma_D \setminus \overline{\omega_2},\\[5pt]
r &=& \xi_2 & \textrm{on}\; \omega_2 \setminus \overline{\omega_1},\\[5pt]
\gamma\frac{\partial r}{\partial n} + k r &=& 0 & \textrm{on}\; \omega_1.
\end{array} \right.
\end{equation}

Since $\xi_2 \geq 0$ on $\overline\Omega$, a straightforward adaptation of the proof of the pointwise bounds \cref{eq.Linftyvphie} shows that:
\begin{equation}\label{eq.rnegmonrDom}
r = \xi_2 - \xi_1 \:\geq \:0 \text{ on }\overline\Omega.
\end{equation}

Let us eventually calculate:
$$ 
\begin{array}{>{\displaystyle}cc>{\displaystyle}l}
\capar^D(\omega_2,k) - \capar^D(\omega_1,k) &=& \int_{\omega_2} \xi_2  \:\d s -  \int_{\omega_1} \xi_1 \:\d s \\[1em]
&=& \int_{\omega_2 \setminus \overline{\omega_1}} \xi_2 \:\d s  + \int_{\omega_1} (\xi_2 - \xi_1) \:\d s.
\end{array}
$$
The first term in the above right-hand side is non negative since $\xi_2 \geq 0$, while the second term in there is also non negative because of \cref{eq.rnegmonrDom}. This terminates the proof of $(ii)$.\par\medskip


\section{Intrinsic definitions of $\capar^N(\omega,k)$ and $\capar^D(\omega,k)$} \label{sec_6}


\noindent In this section, we introduce quantities that are equivalent to the capacities $\capar^D(\omega,k)$ and $\capar^N(\omega,k)$, 
but that are perhaps more natural, as they do not depend on the conductivity $\gamma$, the domain $\Omega$ or the regions $\Gamma_D$ and $\Gamma_N$. 
\medskip

As in the previous sections, $\Omega$ denotes a smooth bounded open 
subset of $\R^d$. Let  $\omega$ be a relatively open subset of $\partial \Omega$ and let $n(x)$ 
be the unit normal vector to $\partial \Omega$ at $ x \in \omega$, pointing outward $\Omega$.
For a quantity $\alpha$ which is smooth from either side of $\omega$ but possibly discontinuous across $\omega$, we denote by 
\begin{equation*}
\forall x \in \partial \Omega, \quad \alpha^\pm (x) := \lim_{t \to 0^+} \alpha (x \pm t n(x)),
\end{equation*}
the one-sided limits of $\alpha$ at $x$. 

We introduce two additional auxiliary functions, $z,w \in H^1(\R^d \setminus \overline{\omega})$,
defined via the following boundary-value problems:
\begin{equation} \label{def_zw}
\left\{ \begin{array}{cccl}
- \Delta z + z &=& 0 &\textrm{in}\; \R^d \setminus \overline\omega,
\\
\frac{\partial z^+}{\partial n} - k z^+ &=& 1 & \textrm{on}\; \omega,
\\
\frac{\partial z^-}{\partial n} + k z^- &=& 1 & \textrm{on}\; \omega.
\end{array}\right.
\quad
\left\{ \begin{array}{cccl}
- \Delta w + w &=& 0 &\textrm{in}\; \R^d \setminus \overline\omega,
\\
\frac{\partial w^+}{\partial n}  - k w^+ &=& -k  & \textrm{on}\; \omega,
\\
\frac{\partial w^-}{\partial n}  + k w^- &=& k  & \textrm{on}\; \omega.
\end{array}\right.
\end{equation}
The associated variational formulations read:
\begin{equation}\label{eq.varfwint}
\text{For all } v \in H^1(\R^d \setminus \overline\omega), \quad \int_{\R^d \setminus \overline{\omega}}  \nabla w \cdot \nabla v  \:\d x + \int_{\R^d} w v \:\d x 
\;+\; k\, \int_{\omega} (w^- v^- + w^+v^+) \:\d s  =  k \, \ds\int_\omega (v^- + v^+) \:\d s,
\end{equation}
and 
$$
\text{For all } v \in H^1(\R^d \setminus \overline\omega), \quad \int_{\R^d \setminus \overline{\omega}} \nabla z \cdot \nabla v \:\d x + \int_{\R^d} z v \:\d x
\;+\; k\, \ds\int_{\omega} (z^- v^- + z^+ v^+) \:\d s = \ds\int_\omega (v^- - v^+) \:\d s.
$$
Their well-posedness is immediately ensured by the Lax Milgram theorem.
\medskip

Let now $\omega$ be an oriented Lipschitz open surface in $\R^d$, and let $k>0$. We define the intrinsic Robin-Neumann capacity $C^N(\omega,k)$ of $\omega$ with parameter $k$ by the formula:
\begin{equation}\label{eq.defCNomk}
 C^N(\omega,k) =  - \int_\omega \left[ \frac{\partial w}{\partial n }\right] \:\d s .
 \end{equation}
Using the Robin boundary conditions featured in the problem \cref{def_zw}, this quantity rewrites:
$$C^N(\omega,k) = k \ds\int_\omega \Big((1-w^-) + (1 - w^+) \Big) \:\d s.$$
Alternatively, invoking the variational form \cref{eq.varfwint} for $w$, we obtain yet another expression of $C^N(\omega,k)$: 
\begin{equation}\label{eq.altCN}
\begin{array}{>{\displaystyle}cc>{\displaystyle}l}
C^N(\omega,k) &=&
\ds\int_{\R^d \setminus \overline{\omega}}
 |\nabla w|^2 \:\d x + \int_{\R^d}w^2 \:\d x \;+\; k\ds\int_{\omega} \Big((w^-)^2 + (w^+)^2\Big)\:\d s \\[1em]
&& \hspace{5cm}\;-\; k \, \ds\int_\omega \Big(w^- + w^+\Big) \:\d s
\;+\; k \ds\int_\omega \Big((1-w^-) + (1 - w^+)\Big)\:\d s
\\
&=&
\ds\int_{\R^d \setminus \overline{\omega}}
 |\nabla w|^2 \:\d x  + \int_{\R^d} w^2 \:\d x \;+\; k\ds\int_{\omega} \Big((1-w^-)^2 + (1-w^+)^2\Big) \:\d s.
\end{array}
\end{equation}

In a similar spirit, we define the intrinsic Robin-Dirichlet capacity $C^D(\omega,k)$ of $\omega$ with parameter $k$ as:
\begin{equation}\label{eq.defCD}
C^D(\omega,k) 
\;=\; 
 \ds\int_\omega (z^- - z^+)\:\d s,
 \end{equation}
 which alternatively equals:
\begin{equation}\label{eq.altCD}
C^D(\omega,k) 
\;=\; 
\ds\int_{\R^d \setminus \overline{\omega}}
 |\nabla z|^2 \:\d x + \int_{\R^d} z^2 \:\d x \;+\; k\, \ds\int_{\omega} \Big((z^-)^2 + (z^+)^2\Big)\:\d s.
\end{equation}
\medskip

Our main result concerning the equivalence between the quantities $\capar^D(\omega,k)$ and the intrinsic version $C^D(\omega,k)$ is the following.

\begin{proposition}\label{prop.eqCD}
Let $\Omega$ be a smooth bounded domain in $\R^d$, whose boundary $\partial \Omega$ is the disjoint reunion of a Neumann and a Dirichlet region $\Gamma_N$, $\Gamma_D$, respectively. Let $\omega \subset \Gamma_D$ be a Lipschitz open subset which is well-inside $\Gamma_D$, in the sense that \cref{eq.assumsepKD} holds true. Let also $k>0$ be a given positive real number. Then there exists a constant $C$ which depends on $\Omega$, $\Gamma_N$, $\Gamma_D$ and the compact set $K_D$, but not on $\omega$ and $k$ such that:
\begin{equation} \label{est_CD}
C^{-1}\,C^D(\omega,k) 
\;\leq\; \capar^D(\omega,k) \;\leq\; 
C\, C^D(\omega,k),
\end{equation}
\end{proposition}
\begin{proof}
Throughout the proof, we denote by $C>0$ a constant, that may change from one line to the other and that depends on $\Omega$, $\Gamma_N$, $\Gamma_D$, the bounds $\gamma^-$ and $\gamma^+$ on the conductivity $\gamma$ and the compact set $K_D$, but that is in any case independent of $\omega$ and $k$. \par\medskip

We first prove the left-hand inequality in \cref{est_CD}. Using the definition \cref{eq.defCD} of $C^D(\omega,k)$ and bringing into play the Robin-Dirichlet equilibrium distribution $\xi \equiv \xi(\omega,k)$ defined by \cref{def_xi} via its Robin boundary condition, we obtain:
\begin{equation*}
\begin{array}{>{\displaystyle}cc>{\displaystyle}l}
C^D(\omega,k) &=&  \int_{\omega} (z^- - z^+)\,\d s \\[1em]
&=& \int_{\omega} \left( \gamma \frac{\partial\xi}{\partial n} + k \xi \right)(z^- - z^+) \,\d s  \\[1em]
&=& \int_{\partial \Omega}\left( \gamma \frac{\partial\xi}{\partial n} + k \xi \right)(z^- - z^+) \,\d s,
\end{array}
\end{equation*}
where the final line follows from the continuity of $z$ across $\partial \Omega \setminus \overline\omega$.

Let us now introduce the extension operator 
$\Eint~: H^1(\R^d \setminus \overline{\Omega}) \to H^1(\Omega)$ which to 
$v \in H^1(\R^d \setminus \overline{\Omega})$ associates the unique $H^1(\Omega)$ solution $\Eint v$ to the boundary-value problem:
\begin{equation}\label{eq.defEext}
 \left\{\begin{array}{cccl}
-\textrm{div}(\gamma \nabla \Eint v) &=& 0 &\textrm{in}\; \Omega,
\\ 
\Eint v &=& v & \textrm{on}\; \partial \Omega.
\end{array}\right.
\end{equation}
It is immediate to see that there exists a constant $C >0$ such that: 
\begin{eqnarray*}
\forall v \in H^1(\R^d \setminus \overline\Omega), \quad
||\Eint v||_{H^1(\Omega)} &\leq& C ||v||_{H^1(\R^d \setminus \overline{\Omega})}.
\end{eqnarray*}
We thus obtain, with the help of the Green's formula:
\begin{equation*}
\begin{array}{>{\displaystyle}cc>{\displaystyle}l}
C^D(\omega,k)  &=&
\int_{\partial \Omega} \gamma \frac{\partial \xi}{\partial n} \Big(z^- - \Eint z^+) \:\d s +  \int_{\omega}k \xi(z^- - z^+) \, \d s \\[1em]
&=&
\int_\Omega \gamma \nabla \xi \cdot \nabla(z^- - \Eint z^+) \:\d x \;+\;
\ds\int_{\omega} k \xi(z^- - z^+) \, \d s.
\end{array}
\end{equation*}
Now using the Cauchy-Schwarz inequality, we arrive at:
\begin{eqnarray*}
C^D(\omega,k)
&\leq& 
|| \nabla \xi||_{L^2(\Omega)^d} ||\nabla(z^- - \Eint z^+)||_{L^2(\Omega)^d}
\;+\; 
k \lvert\lvert \xi \lvert\lvert_{L^2(\omega)} \; ||z^- - z^+||_{L^2(\omega)} 
\\
&\leq&
\left( \ds\int_\Omega \gamma |\nabla \xi|^2 \:\d x + k \ds\int_\omega \xi^2 \:\d s \right)^{1/2}
\;
\left( \ds\int_\Omega |\nabla (z^- - \Eint z^+)|^2 \:\d x + k \ds\int_\omega (z^- - z^+)^2 \:\d s \right)^{1/2}
\\
&\leq& C\, \capar^D(\omega,k)^{1/2} \; 
\left( \ds\int_{\R^d \setminus \overline{\omega}} |\nabla z|^2  \:\d x
+ k \ds\int_\omega \Big( (z^+)^2 + (z^-)^2 \Big) \:\d s \right)^{1/2},
\end{eqnarray*}
where we have used the energy estimate \cref{est_H1ve} to obtain the final line. Now using the alternative expression \cref{eq.altCD} for $C^D(\omega,k)$, we arrive at:
$$
C^D(\omega,k) \: \leq \: C \capar^D(\omega,k),
$$
as desired.
\medskip

We now turn to the right-hand inequality in \cref{est_CD}, recalling the definition \cref{def_xi} of the Robin-Dirichlet equilibrium potential $\xi$, and that \cref{def_cRD} of the corresponding capacity.
Let us introduce a smooth cut-off function $\phi \in \calC_c^\infty(\R^d)$ which equals $1$ on $K_D$ and $0$ on $\Gamma_N$. 
We calculate:
\begin{equation*}
\begin{array}{>{\displaystyle}cc>{\displaystyle}l}
\ds\int_{\omega} \xi \:\d s
&=& \int_{\omega} \xi \left( \frac{\partial z^-}{\partial n} + k z^- \right) \, \d s \\[1em]
&=& \int_{\partial \Omega} \phi \xi \left( \frac{\partial z^-}{\partial n} + k z^- \right) \, \d s \\[1em]
&=&
\ds\int_{\Omega} \phi \xi \, \Delta z \,\d x \;+\;
\ds\int_{\Omega} \nabla (\phi \xi) \cdot \nabla z \,\d x
\;+\;
k  \ds\int_{\partial \Omega} \phi \xi z^- \,\d s \\[1em]
&=&
\ds\int_{\Omega} \Big( \nabla (\phi \xi) \cdot \nabla z  + \phi \xi z \Big) \,\d x
\;+\;
k \ds\int_{\omega} \phi \xi z^- \,\d s,
\end{array}
\end{equation*}
where the final line follows from integration by parts and the problem \cref{def_zw} satisfied by $z$. 

Combining now the Cauchy-Schwarz inequality and the Poincar\'e's inequality, we obtain:
\begin{equation*}
\begin{array}{>{\displaystyle}cc>{\displaystyle}l}
 \int_\omega \xi \:\d s &\leq&
C \, \Big( ||\xi ||_{H^1(\Omega)} \; ||z||_{H^1(\Omega)}
\;+\;  k ||\xi||_{L^2(\omega )} \; ||z^-||_{L^2(\omega)} \Big)
\\
&\leq& 
C \, \left( \ds\int_\Omega \gamma |\nabla \xi|^2 \:\d x  \;+\; k  \ds\int_{\omega } \xi^2 \:\d s  \right)^{1/2}
\; 
\left( ||z||_{H^1(\R^d \setminus \overline{\omega})}^2 \;+\; k  
\ds\int_{\omega } \Big((z^+)^2 + (z^-)^2 \Big) \:\d s \right)^{1/2}.
\end{array}
\end{equation*}
Now using again the formula \cref{def_cRD} for $\capar^D(\omega,k)$ as well as the alternative expression \cref{eq.altCD} for $C^D(\omega,k)$, we obtain:
$$
 \capar^D(\omega,k)  \: \leq \: CC^D(\omega,k),
$$
which is the expected result.
\end{proof}

The corresponding result about the equivalence between Robin-Neumann capacities $\capar^N(\omega,k)$ and $C^N(\omega,k)$ is the following.

\begin{proposition}
Let $\Omega$ be a smooth bounded domain in $\R^d$, whose boundary $\partial \Omega$ is the disjoint reunion of a Neumann and a Dirichlet region $\Gamma_N$, $\Gamma_D$, respectively. Let $\omega \subset \Gamma_D$ be a Lipschitz open subset which is well-inside $\Gamma_D$, in the sense that \cref{eq.assumsepKN} holds true. Let also $k>0$ be a given positive real number. Then there exists a constant $C$ which depends on $\Omega$, $\Gamma_N$, $\Gamma_D$, the conductivity $\gamma$ and the compact set $K_N$, but not on $\omega$ and $k$ such that:
\begin{equation} \label{est_CN}
C^{-1}\,C^N(\omega,k) 
\;\leq\; \capar^N(\omega,k) \;\leq\; 
C\, C^N(\omega,k),
\end{equation}
\end{proposition}

%

\begin{proof}
The proof is quite similar to that of \cref{prop.eqCD} and for brevity, we only provide a sketch. We henceforth denote by $\chi \equiv \chi(\omega,k)$ the Robin-Neumann equilibrium distribution, solution to \cref{def_chi}. 
Again, throughout, $C>0$ is a constant that may change from one line to the other and that depends on $\Omega$, $\Gamma_N$, $\Gamma_D$, $\gamma$ and $K_N$, but that is in any case independent of $\omega$ and $k$. \par\medskip

Let us first prove the left-hand inequality in \cref{est_CN}. To achieve this, we insert the Robin boundary condition satisfied by $\chi$ on $\omega$ into the definition of $C^5\omega,k)$:
$$
\begin{array}{>{\displaystyle}cc>{\displaystyle}l}
C^N(\omega,k) &=& k \int_\omega \Big((1-w^-) + (1 - w^+) \Big) \:\d s \\[1em]
&=&  \int_\omega \left(\gamma \frac{\partial \chi}{\partial n}  + k \chi \right) \Big((1-w^-) + (1 - w^+) \Big) \:\d s \\[1em]
&=& 2 \int_\omega \gamma \frac{\partial \chi}{\partial n} \:\d s -  \int_\omega \gamma \frac{\partial \chi}{\partial n}(w^+ + w^-) \:\d s + \int_\omega k\chi  \Big((1-w^-) + (1 - w^+) \Big)  \:\d s\\[1em]
&=& 2 \int_\omega \gamma \frac{\partial \chi}{\partial n} \:\d s -  \int_\omega \gamma \frac{\partial \chi}{\partial n}(w^+ + w^-) \:\d s - \int_\omega \left[\frac{\partial w}{\partial n} \right] \chi \:\d s .
\end{array}$$
Introducing a smooth cut-off function $\phi$ which equals $1$ on the compact set $K_N$ and $0$ on $\Gamma_D$, and using the definition \cref{eq.defcapar} of $\capar^N(\omega,k)$, we obtain:
\begin{equation}\label{eq.CNint}
C^N(\omega,k)  = 2\capar^N(\omega,k) -  \int_{\partial \Omega} \gamma \frac{\partial \chi}{\partial n} \phi(w^+ + w^-) \:\d s- \int_{\partial\Omega} \left[\frac{\partial w}{\partial n} \right] \chi \phi \:\d s .
\end{equation}
Let us now recall the definition \cref{eq.defEext} of the extension operator $\Eint: H^1(\R^d \setminus \overline{\Omega}) \to H^1(\Omega)$ of a function defined outside $\Omega$ to a function defined inside $\Omega$. Likewise, introducing a fixed bounded domain $\calO\subset \R^d$ such that $\Omega \Subset \calO$, we define the operator $\Eext: H^1(\Omega) \to  H^1(\R^d \setminus \overline{\Omega})$ which to 
$v \in H^1(\Omega)$ associates the unique solution $\Eext v \in H^1_0(\calO)$ to the boundary-value problem:
\begin{equation}\label{eq.defEext}
 \left\{\begin{array}{cccl}
-\textrm{div}(\gamma \nabla \Eext v) &=& 0 &\textrm{in}\; \calO \setminus \overline{\Omega},\\ 
\Eext v &=& v & \textrm{on}\; \partial \Omega \\
\Eext v &=& 0& \text{on } \partial \calO. 
\end{array}\right.
\end{equation}
Again, it is immediate to see that there exists a constant $C >0$ such that: 
\begin{eqnarray*}
\forall v \in H^1(\Omega), \quad
||\Eext v||_{H^1(\R^d \setminus \overline\Omega)} &\leq& C ||v||_{H^1(\Omega)}.
\end{eqnarray*}
We now rewrite \cref{eq.CNint} as:
$$C^N(\omega,k)  = 2\capar^N(\omega,k) -  \int_{\partial \Omega} \gamma \frac{\partial \chi}{\partial n} \phi(\Eint w^+ + w^-) \:\d s- \int_{\partial\Omega}  \gamma \frac{\partial w^+}{\partial n} \Eext (\chi \phi) \:\d s +\int_{\partial\Omega}  \gamma \frac{\partial w^-}{\partial n}  \chi \phi \:\d s . $$
An application of the Green's formula now yields:
$$C^N(\omega,k)  = 2\capar^N(\omega,k) -  \int_{\Omega} \gamma  \nabla \chi \cdot \nabla (\phi(\Eint w^+ + w^-)) \:\d x + \int_{\R^d \setminus \overline\Omega}  \nabla w \cdot \nabla(\Eext (\chi \phi)) \:\d x +  \int_{\Omega}  \nabla w \cdot \nabla(\chi \phi) \:\d x.$$
Now applying the Cauchy-Schwarz inequality, we arrive at:
$$\begin{array}{>{\displaystyle}cc>{\displaystyle}l}
C^N(\omega,k) & \leq&  2\capar^N(\omega,k) + C \lvert\lvert \chi \lvert\lvert_{H^1(\Omega)} \lvert\lvert w \lvert\lvert_{H^1(\R^d \setminus \overline\omega)} \\[1em]
& \leq&  2\capar^N(\omega,k) + C\capar^N(\omega,k) ^{\frac12} C^N(\omega,k)^{\frac12},
\end{array}
$$
where the second line follows from \cref{eq.H1vphie,eq.altCN}.
Solving this quadratic inequality for $C^N(\omega,k)^{1/2} $, we see that there exists a constant $C >0$ such that: 
$$ C^N(\omega,k)  \: \leq \: C \capar^N(\omega,k), $$
which is the desired result. \par\medskip

Let us now turn to the right-hand inequality in \cref{est_CN}. Like in the previous argument, we start with the definition \cref{eq.defcapar} of the Robin-Neumann capacity $\capar^N(\omega,k)$, in which we bring into play the function $w$ in \cref{def_zw} via its boundary condition on $\omega$:
$$
\begin{array}{>{\displaystyle}cc>{\displaystyle}l}
 \capar^N(\omega,k) &=& k \int_\omega(1-\chi) \:\d s \\[1em]
 &=&  -\frac12 \int_\omega \left[\frac{\partial w}{\partial n}\right] (1-\chi) \:\d s  + \frac12 \int_{\omega} k(w^+ + w^-) (1-\chi) \:\d s \\[1em]
  &=&  \frac12 C^N(\omega,k) +\frac12 \int_\omega \left[\frac{\partial w}{\partial n}\right] \chi  \:\d s  + \frac12 \int_{\omega} \gamma \frac{\partial \chi}{\partial n}(w^+ + w^-) \:\d s,
 \end{array}
 $$
 where we have used the definition \cref{eq.defCNomk} of $C^N(\omega,k)$ to obtain the last line. 
 Now introducing a smooth cut-off function $\phi$ which equals $1$ on $\Gamma_N$ and $0$ on $\Gamma_D$, and using again the extension operators $\Eint : H^1(\R^d \setminus \overline\Omega) \to H^1(\Omega)$, 
 $\Eext : H^1(\Omega) \to  H^1(\R^d \setminus \overline\Omega) $, we obtain:
 $$
\begin{array}{>{\displaystyle}cc>{\displaystyle}l}
 \capar^N(\omega,k) &=&  \frac12 C^N(\omega,k) +\frac12 \int_{\partial \Omega} \left[\frac{\partial w}{\partial n}\right]  \phi\chi  \:\d s  + \frac12 \int_{\partial\Omega} \gamma \frac{\partial \chi}{\partial n}(w^+ + w^-) \phi \:\d s \\[1em]
 &=&  \frac12 C^N(\omega,k) +\frac12 \int_{\partial \Omega}  \frac{\partial w^+}{\partial n}  \phi \Eext \chi  \:\d s   - \frac12 \int_{\partial \Omega}  \frac{\partial w^-}{\partial n}  \phi  \chi  \:\d s  + \frac12 \int_{\partial\Omega} \gamma \frac{\partial \chi}{\partial n}(\Eint w^+ + w^-) \phi \:\d s.
 \end{array}
 $$
 Now applying the Green's formula and the Cauchy-Schwarz inequality, we obtain:
 $$  \capar^N(\omega,k) \leq \frac12 C^N(\omega,k)  + C \lvert\lvert w \lvert\lvert_{H^1(\R^d \setminus \overline\omega)} \lvert\lvert \chi \lvert\lvert_{H^1(\Omega)},$$
 and so, considering again \cref{eq.H1vphie,eq.altCN}: 
 $$  \capar^N(\omega,k) \leq \frac12 C^N(\omega,k)  + C \capar^N(\omega,k)^{\frac12} C^N(\omega,k) ^{\frac12}.$$
 Solving this quadratic inequality for $ \capar^N(\omega,k) ^{1/2} $, we see that there exists a constant $C >0$ such that: 
$$   \capar^N(\omega,k) \leq C C^N(\omega,k) ,  $$
as expected. \par\medskip
 \end{proof}

\section{Conclusion and perspectives}

\noindent In this article, we have investigated the asymptotic effect of a particular class of perturbations of the boundary conditions attached to the conductivity equation for the voltage potential: its Neumann or Dirichlet conditions are replaced with a Robin boundary condition with arbitrary admittance parameter $k_\e$ on a subset $\omega_\e$ of the relevant region of the physical domain which vanishes at the limit $\e \to 0$.
In both situations, we have derived a general asymptotic formula for the voltage potential $u_\e$ resulting from such perturbations: the amplitude of the perturbation is quantified by a new quantity depending on $\omega_\e$ and $k_\e$ -- the Robin-Neumann and Robin-Dirichlet capacity, respectively.
We have proved that, in the regimes where $k_\e$ is either ``very small'' or ``very large'', these quantities are consistent with the asymptotic rates obtained in our previous work, where a Neumann boundary condition of the conductivity equation was replaced with a Dirichlet condition on a small subset of $\partial \Omega$, and vice-versa. 
An interesting perspective of this work consists in deriving explicit versions of these asymptotic expansions in the particular physical configuration where the vanishing support $\omega_\e$ of the Robin boundary condition is a curvilinear segment of length $2\e$ in 2d, or a surface disk of radius $\e$ in 3d. Of particular interest are asymptotic formulas that would hold true uniformly with respect to the admittance parameter $k_\e$: these would allow to appraise how the asymptotic rates of convergence of the voltage potential to its background counterpart transition from one extreme behavior (of the order $k_\e \lvert \omega_\e\lvert$ when $k_\e \ll 1$ in the case where the perturbation affects the Neumann boundary) to the other (of the order $\capa(\omega_\e)$ when $k_\e \gg 1$ in the same configuration).

\par\bigskip
\noindent \textbf{Acknowledgements.} 
The work of E.B. and C.D. is partially supported by the project ANR-24-CE40-2216 STOIQUES.
Additionally, C.D. is partially supported by the project ANR-22-CE46-0006 StableProxies.
Part of this article was realized while C.D. was visiting the Department of Mathematics of the University of Trento, whose hospitality is gratefully acknowledged.


\appendix
\section{Control of the extension of a distribution in $H^{-s}(\omega), 0 < s < 1$}\label{sec.app}

\noindent This appendix provides a proof of the fact \cref{est_10b} used during the proof of Theorem~\ref{thm5},
whereby the function $ \frac{\overline{\chi_\e^\infty}}{\partial n}$ lies 
in~$\widetilde{H}^{-s}(\omega_1)$ with an estimate of its norm.
Interestingly, the proof indicates how the estimate may depend on the size of $\omega$. 
Although this degree of precision was not needed in the proof of Theorem~\ref{thm5}, we believe that the result and its proof are interesting in their own right.

\medskip

In this section, we assume that $\Omega$ is a smooth bounded open set in $\R^d, d = 2,3$.
Its boundary is made of two disjoint open Dirichlet and Neumann regions $\Gamma_D$ and $\Gamma_N$ with non-empty interiors:
$$
\partial \Omega \:\: = \:\: \overline{\Gamma_D} \cup \overline{\Gamma_N}, \text{ where } \Gamma_D \cap \Gamma_N = \emptyset.
$$
For simplicity, we assume that $\Gamma_N$ contains the origin $0$ and that it is completely flat in a fixed neighborhood of $0$. The general situation follows from application of a smooth diffeomorphism to this flat configuration.
\medskip

\begin{lemma} \label{lem_inject_tilde}

Let $\omega \Subset \Gamma_N \subset \partial \Omega$ be either a segment of length $2r$ around $0$ if $d=2$, 
or a flat disk with center $0$ and radius $r$ if $d =3$, such that
\begin{equation} \label{assump_lem}
\text{\rm dist}(\omega,\Gamma_D) \:\: \geq \:\: C r,
\end{equation}
for a constant $C >0$ which does not depend on $r$. 
Let $0 < s < 1$, and let $g$ be an element in $H^{-s}(\partial \Omega)$ such that $g = 0$ on 
$\Gamma_N \setminus \overline\omega$.
Then $g$ belongs to $\widetilde{H}^{-s}(\omega)$ and there exists a constant $C >0$ independent of $r$ such that
\begin{equation} \label{est_lem}
||g||_{\widetilde{H}^{-s}(\omega)} 
\:\: \leq \:\: 
\ds\frac{C}{r^s} \, ||g||_{H^{-s}(\partial \Omega)}.
\end{equation}
\end{lemma}
\medskip

\begin{proof}

\noindent For simplicity of the exposition, we only provide the proof when the space dimension equals $d=2$:
$\omega \subset \Gamma_N$ is then the straight segment of length $2r$ around $0$, while $\Gamma_N$ contains the segment of length $2$ around $0$: 
$$
\omega = (-r,r) \times \{0\} \;\subset\;  (-1,1) \times \{0\} \:\subset\: \Gamma_N.
$$ 
With a small abuse of notation, we identify a point $(x,0) \in \omega$ with its horizontal coordinate $x \in (-r,r)$,
and denote by $C$ a positive constant that does not depend on $r$. 
\medskip
Note that, for any Lipschitz subset $\sigma \subset \partial \Omega$, an element $g \in H^{-s}(\partial \Omega)$ naturally induces an element in $H^{-s}(\sigma)$, still denoted by $g$, with
\begin{eqnarray*}
\left\langle g,v\right\rangle_{H^{-s}(\sigma), \widetilde{H}^{s}(\sigma)}
&=&
\left\langle g,\widetilde{v} \right\rangle_{H^{-s}(\partial \Omega), H^{s}(\partial \Omega)},
\end{eqnarray*}
where $\widetilde{v}$ denotes the extension by $0$ of a function $v \in \widetilde{H}^{s}(\sigma)$.
The argument proceeds in two steps.

\medskip

\noindent\textit{Step 1: We construct a suitable extension operator 
from $H^s(\omega)$ to $\widetilde{H}^s(\Gamma_N)$.} 

To achieve this, let $\rho~: \R^+ \rightarrow [0,1]$ denote a smooth function such that
\begin{eqnarray*}
\rho(x) &=& \left\{ \begin{array}{cl}
1 & \textrm{if} \quad 0 \leq x \leq r/2,
\\
0 & \textrm{if} \quad x \geq r,
\end{array}\right.
\end{eqnarray*}
together with the estimate:
\begin{equation} \label{est_rhoep}
\forall x \geq 0, \quad   \lvert \rho^\prime(x) \lvert \:\: \leq\:\: \ds\frac{4}{r}.
\end{equation}
Let $0 < s < 1$. For any function $v \in {\mathcal C}^\infty(\overline{\omega})$, we define:
\begin{eqnarray*}
R v(x) &=& \left\{ \begin{array}{cl}
v(x) & \textrm{if}\quad 0 < x < r,
\\[5pt]
\rho(-x)v(-x) &\textrm{if} \quad x \in \Gamma_N,\quad -r< x < 0,
\\[5pt]
\rho_(x-r)v(2r -x) &\textrm{if} \quad x \in \Gamma_N,\quad r< x < 2r,
\\[5pt]
0 &\textrm{otherwise,}
\end{array} \right.
\end{eqnarray*}
Note that $Rv$ is supported in $\Gamma_N$, due to \cref{assump_lem}, see \cref{fig.graphRe} for an illustration. 

\begin{figure}[!ht]
\centering
\includegraphics[width=0.8\textwidth]{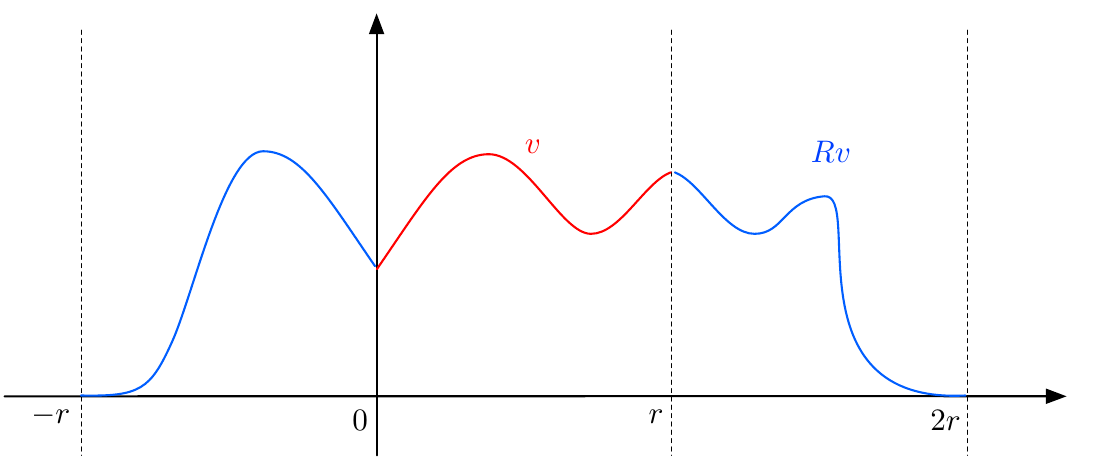}
\caption{\it Illustration of the extension operator $R$ defined in \cref{sec.app}.}
\label{fig.graphRe} 
\end{figure}
\medskip

Recalling the definitions of fractional Sobolev norms on a subset of the boundary of a domain from \cref{sec.Hs}, for an arbitrary $v \in \calC^\infty(\overline\omega)$, we next estimate the norm $ \lvert\lvert R v \lvert\lvert_{\widetilde{H}^s(\Gamma_N)}$ 
in terms of~$ \lvert\lvert v\lvert\lvert_{H^s(\omega)}$:
\begin{eqnarray*}
||v||^2_{H^s(\omega)} &=&
\ds\int_{\omega} |v(x)|^2 \,\d s(x) \;+\;
\ds\int_{\omega}\ds\int_{\omega} \ds\frac{|v(x) - v(y)|^2}{|x-y|^{1 + 2s}} \,\d s(x)\d s(y),
\end{eqnarray*}
and, in view of the definition of $R v$,
\begin{eqnarray} \label{normRv}
||R v||^2_{\widetilde{H}^s(\Gamma_N)} &=&
\ds\int_{\partial \Omega} R v^2 \,\d s(x) \; +\;
\ds\int_{\partial \Omega}\ds\int_{\partial \Omega} 
\ds\frac{|R v(x) - R v(y)|^2}{|x-y|^{1 + 2s}} \,\d s(x)\d s(y)
\nonumber \\
&=&
\ds\int_{\Gamma_N} R v^2 \,\d s(x) \; +\;
\ds\int_{\Gamma_N}\ds\int_{\Gamma_N} 
\ds\frac{|R v(x) - R v(y)|^2}{|x-y|^{1 + 2s}} \,\d s(x)\d s(y).
\end{eqnarray}
\medskip

Let us first estimate
\begin{equation} \label{est_11}
\begin{array}{>{\displaystyle}cc>{\displaystyle}l}
\ds\int_{\Gamma_N} R v(x)^2 \,\d s(x)
&=&
\ds\int_{-r}^0 R v(x)^2 \,\d s(x) 
\;+\; \ds\int_{0}^r R v(x)^2 \,\d s(x) 
\;+\; \ds\int_{r}^{2r} R v(x)^2 \,\d s(x)  \\[1em]
&\leq&
2 \ds\int_0^r \rho(x)^2 v(x)^2 \,\d s(x) + \ds\int_0^r v(x)^2 \:\d s(x) \\[1em]
&\leq& 3 ||v||_{L^2(\omega)}^2,
\end{array}
\end{equation}
where we have used the definition of $R v(x)$ and the fact that $0 \leq \rho(x) \leq 1$.
\medskip

Let us now turn to estimate the $\widetilde H^{s}(\Gamma_N)$ semi-norm of $Rv$; it holds:
\begin{equation}\label{eq.T1T2T3}
\begin{array}{>{\displaystyle}cc>{\displaystyle}l}
\ds\int_{\Gamma_N}\ds\int_{\Gamma_N}  \ds\frac{|R v(x) - R v(y)|^2}{|x-y|^{1 + 2s}} \,\d s(x)\d s(y)
&=&
\ds\int_{\omega}\ds\int_{\omega} 
\ds\frac{|R v(x) - R v(y)|^2}{|x-y|^{1 + 2s}} \,\d s(x)\d s(y) 
\\[1em]
&&
+\; \ds\int_{\Gamma_N \setminus \overline{\omega}}
\ds\int_{\Gamma_N \setminus \overline{\omega}} 
\ds\frac{|R v(x) - R v(y)|^2}{|x-y|^{1 + 2s}} \,\d s(x)\d s(y)
\\[1em] 
&& +\;
2 \, \ds\int_{\Gamma_N \setminus \overline{\omega}}
\ds\int_{\omega} \ds\frac{|R v(x) - R v(y)|^2}{|x-y|^{1 + 2s}} \,\d s(x)\d s(y)
\\[2em]
&=:& T_1 + T_2 + 2 T_3,
\end{array}
\end{equation}
with obvious notations. 

We first estimate the term $T_1$. By construction, $R v(x) = v(x)$ for $x \in \omega$, and so:
\begin{equation} \label{est_12a}
T_1 = \ds\int_{\omega}\ds\int_{\omega} 
\ds\frac{|v(x) - v(y)|^2}{|x-y|^{1 + 2s}} \,\d s(x)\d s(y)  = \lvert v \lvert_{H^s(\omega)}^2.
\end{equation}

We next turn to $T_2$. According to our simplifying assumptions regarding 
the geometry of $\omega$, we have:
\begin{multline*}
T_2 =
\ds\int_{y= -r}^0 \ds\int_{x=-r}^0 
\ds\frac{|\rho(-x) v(-x) - \rho(-y)v(-y)|^2}
{|x-y|^{1 + 2s}} \,\d s(x)\d s(y) 
\\ 
+ \ds\int_{y=r}^{2r} \ds\int_{x=r}^{2r} 
\ds\frac{|\rho(x- 2r) v(2r-x) - \rho(y-2r)v(2r - y)|^2}
{|x-y|^{1 + 2s}} \,\d s(x)\d s(y) 
\\
 + 2 \ds\int_{y=r}^{2r} \ds\int_{x=-r}^{0} 
\ds\frac{|\rho(-x) v(-x) - \rho(y-2r)v(2r - y)|^2}
{|x-y|^{1 + 2s}} \,\d s(x)\d s(y) .
\end{multline*}
We now denote by $S_1$, $S_2$ and $S_3$ the three integrals in the above right-hand side. 
The first term $S_1$ can be estimated as follows:
\begin{equation}\label{eq.estS1}
\begin{array}{>{\displaystyle}cc>{\displaystyle}l}
S_1 &\leq & 
C \ds\int_{y=-r}^0 \ds\int_{x=-r}^0 
\ds\frac{\rho(-x)^2|v(-x) - v(-y)|^2}
{|x-y|^{1 + 2s}} \,\d s(x)\d s(y)  
\\
&& \hspace{5cm} + \:\: C\ds\int_{y= -r}^0 \ds\int_{x=-r}^0 \
\ds\frac{|\rho(-x)  - \rho(-y)|^2 v(-y)^2} {|x-y|^{1 + 2s}} \,\d s(x)\d s(y) 
\\[1em]
&\leq&
C \ds\int_{\omega}\ds\int_{\omega} \ds\frac{|v(x) - v(y)|^2}{|x-y|^{1 + 2s}} \,\d s(x)\d s(y) 
\;+\;  
C \lvert\lvert \rho^\prime\lvert\lvert_{L^\infty(\R_+)}^2
\ds\int_{\omega}\ds\int_{\omega} {|x-y|^{1-2s}}|v(y)|^2 \,\d x\d y 
\\[1em]
&\leq&
C \ds\int_{\omega}\ds\int_{\omega} \ds\frac{|v(x) - v(y)|^2}{|x-y|^{1 + 2s}} \,\d s(x)\d s(y) 
\;+\; 
C ||\rho^\prime||_{L^\infty(\R_+)}^2
\ds\int_{\omega}\ds\int_{\omega} (2 r)^{1-2s} |v(y)|^2 \,\d s(x)\d s(y) 
\\[1em]
&\leq&
C \ds\int_{\omega}\ds\int_{\omega} \ds\frac{|v(x) - v(y)|^2}{|x-y|^{1 + 2s}} \,\d s(x)\d s(y)
\;+\;
C \, r^{-2s} \,\ds\int_{\omega} v(y)^2 \, \d s(y) 
\\[1.5em]
&\leq& C\, r^{-2s} \, ||v||_{H^s(\omega)}^2,
\end{array}
\end{equation}
where the penultimate line follows from \cref{est_rhoep}.
The same argument applies to $S_2$ and shows that: 
\begin{equation}\label{eq.estS2}
 S_2 \leq Cr^{-2s} ||v||_{H^s(\omega)}^2 .
 \end{equation}
As far as $S_3$ is concerned, we have: 
\begin{equation}\label{eq.estS3}
\begin{array}{>{\displaystyle}cc>{\displaystyle}l}
S_3 
&\leq & 
\frac{C}{r^{1+2s}} \ds\int_{y=r}^{2r} \ds\int_{x=-r}^{0} 
\Big(|\rho(-x) v(-x) - \rho(y-2r)v(2r - y)|^2 \Big)\,\d s(x)\d s(y) 
\\[1em]
&\leq & 
\frac{C}{r^{2s}}  \ds\int_{x=0}^{r} 
\lvert v(x) \lvert^2 \,\d s(x)\d s(y) 
\\[1em]
&=& 
C\, r^{-2s}\, \lvert\lvert v \lvert\lvert^2_{L^2(\omega)}.
\end{array}
\end{equation}
Combining the estimates \cref{eq.estS1,eq.estS2,eq.estS3}, we obtain: 
\begin{eqnarray} \label{est_12b}
T_2 &\leq& C\, r^{-2s} \, ||v||_{H^s(\omega)}^2.
\end{eqnarray} 

Finally, let us decompose the last term $T_3$ in \cref{eq.T1T2T3} as follows:
\begin{equation}\label{eq.T3}
\begin{array}{>{\displaystyle}cc>{\displaystyle}l}
T_3 &=&
\ds\int_{\Gamma_N \setminus \overline{\omega}}\ds\int_{\omega} 
\ds\frac{|R v(x) - R v(y)|^2}{|x-y|^{1 + 2s}} \,\d s(x)\d s(y),
\\[1em]
&=:& R_1 + R_2,
\end{array}
\end{equation}
where
$$
R_1 = \ds\int_{y = -r}^{0} \ds\int_0^r \ds\frac{|v(x) - R v(y)|^2}{|x-y|^{1 + 2s}} \,\d s(x)\d s(y)
\text{ and } 
R_2 =  \ds\int_{y = r}^{2r} \ds\int_0^r \ds\frac{|v(x) - R v(y)|^2}{|x-y|^{1 + 2s}} \,\d s(x)\d s(y).
$$ 
In view of the definition of $\rho$, and since $|x-y| > r/2$ when
$-r < y < -r/2 < 0 < x$, we see that
\begin{eqnarray*}
R_1 &=&
\ds\int_{y = -r}^{-r/2} \ds\int_0^r
\ds\frac{|v(x) - R v(y)|^2}{|x-y|^{1 + 2s}} \,\d s(x)\d s(y)
\;+\;
\ds\int_{y = -r/2}^{0} \ds\int_0^r
\ds\frac{|v(x) - \rho(-y)v(-y)|^2}{|x-y|^{1 + 2s}} \,\d s(x)\d s(y)
\\
&\leq&
C \ds\int_{y = -r}^{-r/2} \ds\int_0^r
\ds\frac{|v(x) - R v(y)|^2}{r^{1 + 2s}} \,\d s(x)\d s(y)
\;+\;
\ds\int_{y = -r/2}^{0} \ds\int_0^r
\ds\frac{|v(x) - v(-y)|^2}{|x-y|^{1 + 2s}} \,\d s(x)\d s(y),
\end{eqnarray*}
where we have used the definition of $\rho$ to obtain the last line. 
Rearranging this expression, we obtain:
\begin{eqnarray*}
R_1 &\leq&
C \ds\int_{y = -r}^{-r/2} \ds\int_0^r
\ds\frac{1}{r^{1 + 2s}} \big( \lvert v(x) \lvert ^2 + \lvert v(-y) \lvert ^2\big) \,\d s(x)\d s(y)
\;+\;
\ds\int_{y = -r/2}^{0} \ds\int_0^r
\ds\frac{|v(x) - v(-y)|^2}{|x-y|^{1 + 2s}} \,\d s(x)\d s(y)
\\
&\leq&
C r^{-2s} \, ||v||_{L^2(\omega)}^2 + 
\ds\int_{\omega}\ds\int_{\omega}
\ds\frac{|v(x) - v(y)|^2}{|x-y|^{1 + 2s}} \,\d s(x)\d s(y),
\\
\end{eqnarray*}
where we have used the fact that $|x+y| < |x-y|$ 
when $-r/2 < y < 0 < x$ to pass from the first line to the second one.
 
The second term $R_2$ in the expression~\cref{eq.T3} of $T_3$ 
can be estimated in the same manner, which finally yields:
\begin{equation} \label{est_12c}
T_3  \:\: \leq \:\: 
C\, r^{-2s} \, ||v||_{H^s(\omega)}^2.
\end{equation} 
Gathering~\cref{est_11}-\cref{est_12c} and noticing that $2r = |\omega|$, we have eventually proved that:
\begin{equation} \label{est_13}
\forall v \in \calC^\infty(\overline\omega), \quad ||R v||_{\widetilde{H}^s(\Gamma_N)}
\:\: \leq  \:\: C \, |\omega|^{-s} \, ||v||_{H^s(\omega)}.
\end{equation}
By density of $\calC^\infty(\overline\omega)$ in $H^s(\omega)$, this inequality allows to extend $R$ as a bounded linear mapping from $H^s(\omega)$ to $\widetilde H^s(\Gamma_N)$, which still satisfies the estimate \cref{est_13}. 
\par\medskip

\noindent \textit{Step 2: We estimate the $\widetilde{H}^{-s}(\omega)$ 
norm of $g$ by duality.} 

Let $0 < s < 1$, and let $g$ denote a function in $H^{-s}(\partial \Omega)$ such that
$$
g = 0 \quad \textrm{on}\; \Gamma_N \setminus \overline{\omega}.
$$
Furthermore, let $v \neq 0$ be a given function in $H^s(\omega)$. Using the
operator $R$ constructed in step 1, we may extend $v$ to a function $Rv$ defined on
the whole of $\partial \Omega$ such that $Rv = 0$ on $\Gamma_D$. It follows that:
\begin{equation*}
\begin{array}{>{\displaystyle}cc>{\displaystyle}l}
\left\langle g,v \right\rangle_{\widetilde H^{-s}(\omega), H^s(\omega)}
&=& \left\langle g, R v\right\rangle_{H^{-s}(\Gamma_N),\widetilde{H}^s(\Gamma_N)}
\\[1em]
&=& ||R v||_{\widetilde{H}^s(\Gamma_N)}
\left\langle g, \ds\frac{R v}{||R v||_{\widetilde{H}^s(\Gamma_N)}}
\right\rangle_{H^{-s}(\Gamma_N),\widetilde{H}^s(\Gamma_N)}
\\[2em]
&\leq & 
||R||_{{\mathcal L}(H^s(\omega),\widetilde{H}^s(\Gamma_N))}\, ||v||_{H^s(\omega)} 
\, \langle g,w \rangle_{H^{-s}(\Gamma_N),\widetilde{H}^s(\Gamma_N)},
\end{array}
\end{equation*}
where $w := \ds\frac{R v}{||R v||_{\widetilde{H}^s(\Gamma_N)}}$.
Taking the supremum with respect to $w \in \widetilde{H}^s(\Gamma_N)$,
then with respect to $v~\in~H^s(\omega)$,
we obtain 
$$
\begin{array}{>{\displaystyle}cc>{\displaystyle}l}
\sup_{ v \in H^s(\omega),\atop
||v||_{H^s(\omega)} = 1}
\left\langle g,v \right\rangle_{H^{-s}(\partial \Omega), H^s(\partial \Omega)}
&\leq &
||R||_{{\mathcal L}(H^s(\omega),\widetilde{H}^s(\Gamma_N))}
\sup_{ 
w \in \widetilde{H}^s(\Gamma_N) ,\atop 
||w||_{\widetilde{H}^s(\Gamma_N)} = 1 }
\langle g,w \rangle_{H^{-s}(\Gamma_N),\widetilde{H}^s(\Gamma_N)} \\[1em]
&\leq &
||R||_{{\mathcal L}(H^s(\omega),\widetilde{H}^s(\Gamma_N))}
\sup_{ 
w \in \widetilde{H}^s(\partial\Omega) ,\atop 
||w||_{\widetilde{H}^s(\partial \Omega)} = 1 }
\langle g,w \rangle_{H^{-s}(\partial \Omega),H^s(\partial \Omega)}.
\end{array}
$$
In other words, $g$ belongs to $\widetilde{H}^{-s}(\omega)$ and we have:
$$||g||_{\widetilde{H}^{-s}(\omega)}
\:\: \leq\:\: 
||R||_{{\mathcal L}(H^s(\omega),\widetilde{H}^s(\Gamma_N))}
\, ||g||_{H^{-s}(\partial \Omega)}.
$$
Recalling the estimate \cref{est_13}, we arrive at the desired conclusion \cref{est_lem}, which terminates the proof. 
\end{proof}


\bibliographystyle{siam}
\bibliography{asymp_Robin_9.bib}
\end{document}